\documentclass[10pt, a4paper, reqno]{amsart}

\usepackage[utf8]{inputenc}

\usepackage[english]{babel}

\usepackage{amsmath}
\usepackage{amsfonts}
\usepackage{amssymb}
\usepackage{amsthm}
\usepackage[foot]{amsaddr}
\usepackage{mathtools}
\usepackage{bm, bbm}
\usepackage{bigints}
\usepackage{hyperref}
\usepackage[capitalise]{cleveref}
\usepackage{xparse}
\usepackage{graphicx}
\usepackage[percent]{overpic}
\usepackage[usenames,dvipsnames,svgnames,table]{xcolor}
\usepackage[scr=boondoxo, scrscaled=.98]{mathalfa}
\usepackage{dsfont}

\usepackage{enumerate
}

\usepackage[numbers]{natbib}

\nocite{*}

\usepackage[left=3cm, right=3cm, top=3.5cm, bottom=4cm]{geometry}

\DeclareDocumentCommand{\bmd}{O{s} O{y} O{0} O{t} O{u}}
{P \Big\{ B(#1) \in \mathrm d #2 \,\Big \vert \min_{#3\leq z \leq #4} B(z)> 0 , B(#3)=#5 \Big\}}

\DeclareDocumentCommand{\bmdd}{O{s} O{y} O{0} O{t} O{u}}
{P \Big\{ B^\mu(#1) \in \mathrm d #2 \,\Big \vert \min_{#3\leq z \leq #4} B^\mu(z)> 0 , B^\mu(#3)=#5 \Big\}}

\DeclareDocumentCommand{\bmdv}{O{s} O{y} O{0} O{t} O{u}}
{P \Big\{B(#1) \in \mathrm d #2 \,\Big \vert \min_{#3\leq z \leq #4} B(z)> v , B(#3)=#5 \Big\}}

\DeclareDocumentCommand{\bmdb}{O{s} O{y} O{0} O{t} O{u} O{v} }
{P \Big\{B(#1) \in \mathrm d #2 \,\Big \vert \min_{#3\leq z \leq #4} B(z)> 0 , B(#3)=#5, 
	B(#4)=#6 \Big\}}

\DeclareDocumentCommand{\bmdbd}{O{s} O{y} O{0} O{t} O{u} O{v} }
{P \Big\{B^\mu(#1) \in \mathrm d #2 \,\Big \vert \min_{#3\leq z \leq #4} B^\mu(z)> 0 , 	
	B^\mu(#3)=#5, 
	B^\mu(#4)=#6 \Big\}}

\newcommand\numberthis{\addtocounter{equation}{1}\tag{\theequation}}
\numberwithin{equation}{section}

\allowdisplaybreaks

\theoremstyle{remark}
\newtheorem{remark}{Remark}[section]

\theoremstyle{plain}
\newtheorem{theorem}{Theorem}[section]
\newtheorem{lemma}{Lemma}[section]
\newtheorem{prop}{Proposition}[section]

\theoremstyle{definition}
\newtheorem{corollary}{Corollary}[section]

\theoremstyle{definition}
\newtheorem{definition}{Definition}[section]

\title{SOME RESULTS ON THE BROWNIAN MEANDER WITH DRIFT}
\author{F. Iafrate and E. Orsingher}
\address{Sapienza, University of Rome, Italy.}
\email{enzo.orsingher@uniroma1.it, francesco.iafrate@uniroma1.it}
\date{\scriptsize \texttt{\today}}

\begin{document}

\begin{abstract}
In this paper we study the drifted Brownian meander, that is a Brownian motion starting from $ u $
and subject to the condition that $ \min_{ 0\leq z \leq t} B(z)> v  $ with $  u > v $. 
The limiting process for $ u \downarrow v $ is analyzed and the sufficient conditions for its construction
are given. 
We also study the distribution of the maximum of the meander with drift and the related first-passage times. 
The representation of the meander endowed with a drift is provided and extends the well-known result of the 
driftless case. 
The last part concerns the drifted excursion process the distribution of which coincides with the driftless case. 
\end{abstract}
\maketitle

\keywords{\small \textbf{Keywords}: Tightness, Weak convergence, First passage times, Absorbing drifted Brownian motion, Drifted Brownian Excursion
}

\section{Introduction}

The Brownian meander is a Brownian motion 
$\{B(t)\,,\,\, t>0\}$
evolving under the condition that $ \min_{0\leq s\leq t} B(s) > 0 $. 
If the additional condition that $ B(t) = c >0  $ then we have the Brownian excursion which 
is the bridge of the Brownian meander.


Early results in this field emerged in the study of the behaviour of 
random walks conditioned to stay positive where the Brownian meander 
was obtained as the weak limit of such conditional processes (\cite{belkin70}, 
\cite{iglehart74}). In the same spirit the distribution of the maximum 
of the Brownian meander and excursion has been derived in \cite{kaigh78}. 
Further investigations about such distributions can be found in \cite{chung1976}. 
Some important results can also be found in the classical book by 
\citet{ito1996diffusion}. The notion 
of Brownian meander as a conditional Brownian motion and 
problems concerning weak convergence to such processes have been treated 
in \cite{durrett77}. 
More recently analogous results have been obtained in the general setting
of Lévy processes (see for example \cite{chaumont2005}). 

Brownian meanders emerge in path decompositions of the Brownian motion.
In particular Denisov (\cite{denisov84})  shows that a Brownian motion around a maximum 
point can be represented (in law) by means of a two-sided Brownian meander, which is constructed by gluing together 
two meanders. 

These processes also arise in several scientific fields. 
Possible applications range from SPDE's with reflection (\cite{zambotti04}) to enumeration of random graphs (see \cite{janson07} for a survey of the results in this field). 

In this paper we study the meandering process of the drifted Brownian motion $ B^\mu (t) $ and, 
in particular we start by analyzing the joint $ n- $fold distributions
\begin{equation}\label{eq:mdr-joint-intro}
P \bigg\{ \bigcap_{j=1}^n \left( B^\mu(s_j) \in \mathrm d y_j \right) \,\Big \vert \min_{0\leq z \leq t} B^\mu(z)> v , B^\mu(0)=u \bigg \} 
\end{equation}
for $ v < y_j,\,0< s_j \leq t,\, j= 1, \ldots n $ and $u>v $.

In order to examine the interesting case where the starting point coincides with the barrier level $ v $
we use the tools of weak convergence of probability measures. In section 2 we briefly recall the needed results and
in section 3 we show that the required conditions hold. 

We will study some sample path properties of the Brownian meander regarding the maximal oscillation around the starting point 
\begin{equation}\label{eq:lemma-cond-intro}
\lim_{\delta \to 0}
\lim_{u \to v} 
P \Big( \max_{0\leq z \leq \delta } | B^\mu ( z) - B^\mu(0)  | < \eta  \, \Big | 
\min_{ 0\leq z \leq t} B^\mu ( z) > v ,  B^\mu(0) = u \Big) = 1 \qquad \forall \eta > 0
\end{equation}
and the maximal oscillation of the sample paths at an arbitrary point $ s $
\begin{align}
& \lim_{\delta \to 0} \lim_{ u \downarrow v} P \left\{  \max_{s - \delta \leq z \leq s +\delta } |B^\mu (z) | \leq \eta \,\Big | \min_{ 0 \leq z \leq t } B^\mu(z) > v, B^\mu (0) = u\right\} 
\\
&= 
P\left\{ B^\mu (s ) \leq \eta \Big  | \inf_{ 0 < z < t } B^\mu(z) > v, B^\mu(0) = v \right\} \, .
\notag 
\end{align}

In the limit for $ u \to v $ we obtain a stochastic process with marginal distributions equal to 
\begin{align}\label{eq:mdr-0-dist-intro}
&
P \bigg\{  
B^\mu ( s) \in \mathrm d y \bigg |  \inf_{ 0 < z < t } B^\mu( z)  > v, B^\mu(0) = v  
\bigg \} \\
&=
\mathrm d y \left(\frac ts \right)^{\frac 32} 
\frac{
	(y - v) e^{  -\frac{(y-v)^2}{2s}  }
}{
	\int_0^{\infty} w  e^{  - \frac { w^2 } { 2t } + \mu w } \,\mathrm d w 
}
\int_0^\infty  \left( e^{ - \frac{(w-(y-v) )^2}{2 ( t-s)}}  - e^{ - \frac{(w + (y - v ) )^2}{2 ( t-s)}}  \right)
	\frac{e^{\mu w }}{\sqrt{2 \pi( t-s)}} \, \mathrm d w \notag 
\end{align}
$ s < t, \, y > v $. In particular, for $ s = t $, \eqref{eq:mdr-0-dist-intro} reduces to 
\begin{align*}
&
P \bigg\{  
B^\mu ( t) \in \mathrm d y \bigg |  \inf_{ 0 < z < t } B^\mu( z)  > v, B^\mu(0) = v  
\bigg \} 
=
\frac{
	(y - v )   e^{  - \frac{ (y-v)^2}{2t} } e^{\mu y } 
}{
	\int_v^\infty (w-v) e^{  - \frac{ (w-v)^2}{2t}   }  e^{ \mu w} \, \mathrm d w
} \,\mathrm d y 
\numberthis \label{eq:bmd-drift-dist-t-intro} \,\, , \qquad y > v\,\, .
\end{align*}
which for $ \mu = 0 $ coincides with the truncated Rayleigh distribution. 
We obtain the distribution of the maximum of the Brownian meander which, in the simplest case, has the form
\begin{align*}
&
P\bigg\{ \max_{0 \leq z \leq t} B^\mu(z) <x \Big |  \inf_{0 < z < t} B^\mu (z) > 0, B^\mu(0) = 0  \bigg\}
\label{eq:max-bmd-t-v0-fin-intro}
\numberthis
\\
&=
\frac{  
	\sum_{r = -\infty} ^ { + \infty}
	(-1)^r
	e^{  -  \mu r x - \frac{x^2 r^2}{2t}  } 
	+ \mu
	\int_0^\infty e^{ - \frac{w^2}{2t} } e^{ \mu (-1)^{ \left \lfloor \frac w x \right \rfloor} } \,\mathrm d w
}{
	1 + \mu \int_0^\infty 
	e^{  - \frac{y^2}{2t}  } 
	e^{ \mu y } \,\mathrm d y	
}
\end{align*}
and for $  \mu = 0 $  yields the well-known distribution 
\begin{equation}\label{eq:max-bmd-mu0-intro}
P\bigg\{ \max_{0 \leq z \leq t} B(z) <x \, \Big |  \inf_{0 < z < t} B (z) > 0, B(0) = 0  \bigg\}
=
\sum_{r = - \infty} ^ { + \infty} 
(-1)^r
e^{  - \frac{x^2 r^2 }{2t}  } \,\, .
\end{equation}
A related result concerns the first-passage time of the drifted meander. The random variable 
$ T_x = \inf \{ s < t' : B^\mu(s) = x  \} $, $ t'> t $, under the condition that $ \min_{ 0 \leq z \leq t } B^\mu (z) > v $ 
has distribution $ P(T_x > s | \min_{ 0 \leq z \leq t } B^\mu (z) > v, B^\mu(0) = u) $ with a substantially 
different structure for $ s < t  $ and $ t < s < t' $. The effect of the conditioning event in the case $ s > t $ 
corresponds to assuming a Rayleigh-distributed starting point at time $ t $.  

The fifth section of the paper is devoted to the extension to the drifted Brownian meander of the 
relationship
\begin{equation}\label{eq:mdr-repr-0-intro}
\frac{
	\Big|  B(T_0 + s (t - T_0)) \Big|
}{\sqrt{t - T_0}}
\stackrel{i.d.}{=} M(s) \qquad 0 < s < 1
\end{equation}
where $ T_0 = \sup \{s < t : B(s) = 0\} $ (see \citet{pitman99}). 
We are able to show  that 
\begin{align}\label{eq:mdr-repr-intro}
P\left\{   
\frac { 
	\Big | B^\mu\Big ( T_0^\mu + s ( t - T_0^\mu) \Big )  \Big |    
}
{
	\sqrt{t - T^\mu_0}
} \in \mathrm d y 
\right\} = 
\frac 12 
\mathbb E \left[
P\left\{ M^{- \mu \sqrt{ t - T_0^\mu}}(s) \in\mathrm d y \right\} + 
P\left\{ M^{\mu \sqrt{ t - T_0^\mu}} (s)\in\mathrm d y \right\} 
\right]
\end{align}
$ 0 < s < 1 $, where  $M^{\pm \mu \sqrt{ t - T_0^\mu}}$ is the Brownian meander starting at zero
with a time-varying drift $ \pm \mu \sqrt{ t - T_0^\mu} $ and $ T_0^\mu = \sup\{s < t: B^\mu(s) = 0\} $.

Result \eqref{eq:mdr-repr-intro} can also be written as
\begin{equation}
	\frac{
		\Big|  B^\mu(T_0 + s (t - T^\mu_0)) \Big|
	}{\sqrt{t - T^\mu_0}}
	\stackrel{i.d.}{=} M^{\mu X \sqrt{ t - T_0^\mu} }(s) \qquad 0 < s < 1
\end{equation}
where $ X $ is a r.v. taking values $ \pm 1 $ with equal probability, independent from $ M $ and $ T^\mu_0 $ . 
In the last section we give the distribution of the Brownian excursion and show that
	\begin{align*}
	& 
	P\Big\{ B^\mu (s ) \in \mathrm d y \Big | \inf_{ 0 < z < t } B^\mu ( z)> 0, B^\mu ( 0) =  B^\mu ( t) = 0 \Big\}
	\numberthis \label{eq:exc-intro} \\
	&=
	\sqrt \frac{2}{\pi} y^2 \left(\frac{t}{s(t-s)}\right) ^\frac 32 
	e^{- \frac{ y^2t}{ 2 s (t-s)}} \, \mathrm d y \qquad y>0 \,,\,\, s<t
	\\
	&=
	P\Big\{  B(s ) \in \mathrm d y \Big | \inf_{ 0 < z < t } B ( z)> 0, B ( 0) =  B ( t) = 0 \Big\} 
	\end{align*}
which therefore does not depend on the drift $ \mu $.
%
%
%


\section{Preliminaries}\label{sec:prelim}
Let $ \{ B(t),  t \in [0,T] \}$ be a Brownian motion adapted to the natural filtration on some measurable space $ (\Omega, \mathscr F) $
and let $ \{ P_u , u \in \mathbb R \} $ be a family of probability measures such that, under each $ P_u $, $ B $ is a Brownian motion and $  P(B(0) = u) = 1$. 
We consider a drifted Brownian motion $ \{ B^\mu (t), 0 \leq t \leq T\}  $ 
defined as $ B^\mu ( t) = B(t) + \mu t \, , 0 \leq t \leq T  $, 
with $ \mu \in \mathbbm{R} $.
The space $ C[0,T] $ of its sample paths, sometimes indicated as $ \omega = \omega(t) $,
is endowed with the Borel $ \sigma $-algebra $ \mathscr C $ generated by the open sets induced 
by the supremum metric. 

For a given probability space $ (\Omega, \mathscr F, P) $
we define the random function 
\begin{equation}\label{eq:rand-fun}
Y: (\Omega, \mathscr F) \mapsto (C[0,T], \mathscr C) \,\, .
 \end{equation}
We take a probability measure $ \mu $ on $ (C[0,T], \mathscr C) $ defined as
\begin{equation}
\mu(A) = P(Y^{-1} ( A)) \qquad A \in \mathscr C \,\, .
\end{equation}
%
%

For a set $ \Lambda \in \mathscr C $ such that $ \mu(\Lambda) > 0 $  we consider the space 
$
( \Lambda,  \mathscr C, \mu(\,\cdot \, | \Lambda ))
$
%
which is the trace of $ (C[0,T], \mathscr C, \mu) $ on the set $ \Lambda $,  where 
the conditional  probability measure $  \mu(\,\cdot \, | \Lambda ) : \Lambda \cap \mathscr C \mapsto [0,1]  $
is defined in the usual sense as
\begin{equation}
\mu(A  | \Lambda )= \frac{\mu(A \cap \Lambda )}{\mu(\Lambda)} \qquad A \in  \mathscr C \,\, .
\end{equation}
We then construct the space 
$
\big ( Y^{-1} (\Lambda),\mathscr F \cap  Y^{-1} (\Lambda), P ( \,\cdot \, | Y^{-1}(\Lambda) ) \big )
$
where
\begin{equation}
	P ( A | Y^{-1}(\Lambda) ) = \frac{P(A \cap Y^{-1} (\Lambda ) )}{P(Y^{-1} (\Lambda ) )} \qquad 
	\text{ for } A 
	\in \mathscr  F \cap  Y^{-1} (\Lambda) \,\, .
\end{equation}
\begin{definition}
	Given a random function $ Y  $ as in \eqref{eq:rand-fun} and a set $ \Lambda \in \mathscr C $ the \emph{conditional process} $ Y| \Lambda $ is defined  
	as the restriction of $ Y $ to the set $ \Lambda $:
	\begin{equation}
	Y|\Lambda : 
	\big (
	\,  Y^{-1} (\Lambda), \mathscr F  \cap  Y^{-1} (\Lambda), P (\, \cdot \,  | {Y^{-1}(\Lambda)} )  \, \big )
	\mapsto
	( \Lambda, \mathscr C, \mu(\, \cdot \,| \Lambda )  ) \,\, 
	\end{equation} 
\end{definition}
%
%

%
%

The following lemma provides the conditions for a conditional process to be Markov (see \cite{durrett77}). 

\begin{lemma}\label{lem:markov-cond}
	Let $ Y $ be a Markov process on $ C[0,T] $ and let $ \Lambda \in \mathscr C $ such that $ \mu(\Lambda) > 0 $. 
	Let $ \pi_{[0,t]} $ and  $ \pi_{[t,1]} $ be the projection maps on $ C[0,T] $ onto  $ C[0,t] $ and 
	 $ C[t,1] $, respectively. If for all $ t \in [0,T] $ there exist sets $ A_t \in \mathscr B(C[0,t] )$ and $ B_t \in \mathscr B(C[t,1] ) $ such that 
	 $ \Lambda   =   \pi_{[0,t]} ^{-1} A_t  \cap \pi_{[t,1]} ^{-1} B_t$ then $ Y|\Lambda  $ is Markov, 
	 where $ \mathscr B $ denotes the Borel sigma-algebra. 
\end{lemma}

In the following $ \mu(\,\cdot \,) $ denotes  the Wiener measure on $ (C[0,T], \mathscr C) $. 
For a Brownian motion starting at $ u $ we usually write $ P(\,\cdot \, | B(0) = u ) $  to denote $ P_u(\,\cdot \,) $
to underline the dependence on the starting point. 
The \emph{ drifted Brownian meander} can be represented as a conditional process $ B^\mu | \Lambda_{u,v} $ where 
 the conditioning event $ \Lambda_{u,v} $ is of the form
\[ \Lambda_{u,v} = \Big\{  \min_{ 0\leq z \leq t} B^{\mu}(z)  > v, B^\mu(0) = u\Big\} \, .\]
Analogously the \emph{Brownian excursion} is a conditional process $ B^\mu | \Lambda_{u,v,c}  $ with
\[ \Lambda_{u,v,c} = \Big\{ \min_{ 0\leq z \leq t} B^{\mu}(z)>v,\, B^\mu(0) = u, B^\mu(t) = c \Big\} 
\qquad u,c > v \,\, .
\]

We remark that the conditional processes introduced above are Markovian in light of \autoref{lem:markov-cond}.

For some fixed $ v > 0 $, we need to study the weak convergence of the measures 
$ \mu_{u,v} := \mu( \, \cdot \, | \Lambda_{u,v}) $ as $ u \downarrow v $. 
See \citet{billingsley2009convergence} for a treatise of the general theory of weak convergence. We here recall the main 
concepts we will make use of.

\begin{definition}
	Given a metric space $ (S, \rho) $ and a family $ \Pi $ of probability measures on $ (S, \mathscr B ( S)) $, 
	 $ \mathscr B ( S) $ being the Borel $ \sigma- $field on $ S $, we say that $ \Pi $ is tight if
	 \[
	 \forall \eta > 0 \quad  \exists \text{ compact } K \subset S  \quad  \text{s.t.} \quad \forall \mu \in \Pi \quad  \mu(K) > 1 - \eta 
	 .
	 \]
\end{definition}

The tightness property  is equivalent to  relative compactness if $ (S, \rho) $ is separable and complete,
as estabilished in a well-known theorem due to Prohorov, and it is thus relevant to prove the weak convergence of measures. 
In fact the following theorem holds (see \cite{billingsley2009convergence}, Theorem 7.1, or \cite{karatzas2014brownian}, Theorem 4.15).
\begin{theorem}\label{thm:weak-conv}
	Let $ \{ X^{(n)} \}_n $ and $ X $ be stochastic processes on some probability space $ (\Omega, \mathscr F, P) $ onto $(C([0,T]), \mathscr C )$ and let $ \{\mu_n \}_n $ and $ \mu $, respectively, the induced measures. 
	If for every $ m $ and for every $ 0 \leq t_1 < t_2 < \cdots t_m \leq t $, 
	the finite dimensional distributions of $ (X^{(n)}_{t_1}, \ldots, X^{(n)}_{t_m} ) $ converge 
	to those of $ (X_{t_1}, \ldots, X_{t_m} ) $ and the family $ \{\mu_n\}  $ is tight then $ \mu_n \Rightarrow \mu $.
\end{theorem}
In the following section we will compute the limit of the finite dimensional distributions. 
As far as the tightness is concerned we will make use of the following theorem 
(\cite{billingsley2009convergence}, Theorem 7.3)
which characterizes the tightness
of a family of measures induced by a process with a.s. continuous paths 
in terms of its modulus of continuity $ m_\omega^T (\delta) = \sup_{s,t \in [0,T]: |t-s| < \delta} |\omega(s) - \omega(t)| $.

\begin{theorem}\label{thm:tight}
	A sequence of probability measures $ \{P_n\}_n $ on $ (C[0,T], \mathscr C) $ is tight if and only if
	\begin{enumerate}[(i)]
		\item  $ \qquad  \forall \eta > 0 \,\,  \exists a  $ and $  N $ such that 
		$ \displaystyle P_n(\omega :  |\omega(0)| > a  )  < \eta, \,\, \forall n \geq  N ,  $ and 
		\item  $ \qquad  \displaystyle \forall \eta > 0 \quad  
		 \lim_{\delta \downarrow 0} \limsup_{n \to \infty } P_n\{ \omega: m_\omega^T ( \delta) \geq \eta   \} = 0.
		$	
	\end{enumerate}
\end{theorem}
%
%
%


Condition $ (ii) $ is difficult to prove directly so we use the Kolmogorov-Cent\v sov theorem 
which necessitates bounds on the expectation of the increments. 
We first recall the Kolmogorov-\v Centsov theorem (\cite{karatzas2014brownian}, Th. 2.8).
\begin{theorem}\label{thm:kolmo-cen}
Suppose that a process $ X = \{X(t), 0 \leq t \leq T \} $ on a probability space $ (\Omega, \mathscr F, P) $ satisfies the condition 
\begin{equation}\label{eq:kolmo-cen-cond}
\mathbb E | X(t) - X(s) | ^\alpha \leq C | t-s|^{1 + \beta} \,\,,\quad 0 \leq s,t \leq T
\end{equation}
for some positive constants $ \alpha, \beta $ and $ C $. Then there exists a continuous modification of $ X $ which is locally H\"older continuous 
with exponent $ \gamma \in (0, \beta / \alpha) $ i.e., 
\begin{equation}
P\left\{  \omega : \sup_{ \substack{  s,t \in [0,T] \\ |t-s| < h(w)} }
\frac{  | \tilde X(\omega, t)    -   \tilde X(\omega, s)  |}{|t-s|^\gamma}  
\leq \delta 
\right\} = 1
\end{equation}
where $ h $ is an a.s. positive random variable and $ \delta > 0 $ is an appropriate constant. 
\end{theorem} 
The  Kolmogorov-\v Centsov theorem can be exploited to prove the tightness property of a family of measures as 
stated in the following result (\cite{karatzas2014brownian}, Problem 4.11) 
\begin{prop}\label{pr:kolmo-cen-tight}
	Let $ \{ X^{(m)}, m\geq 1 \} $ be a sequence of stochastic processes $ X^{(m)} = \{X^{(m)}(t), 0 \leq t \leq T \} $
	on $ (\Omega, \mathscr F, P) $, satisfying the following conditions
	\begin{enumerate}[(i)]
		\item  $ \qquad  \sup_{m \geq 1 } \mathbb E |X_0^{(m)}|^\nu < \infty $, 
		\item  $ \qquad \sup_{m \geq 1 }  \mathbb E | X^{(m)}(t) - X^{(m)}(s) | ^\alpha \leq C_T | t-s|^{1 + \beta} \,\,,\quad 0 \leq s,t \leq T $
	\end{enumerate}
for some positive constants $ \alpha, \beta, \nu $ and $ C_T $ (depending on $ T $). Then the probability measures 
$ P_m = P(X_m^{-1}) \,, m\geq 1$  induced by these processes form a tight sequence. 
\end{prop}
It is easily seen that conditions $ (\emph{i}) $  and $ (\emph{ii}) $ of Proposition \autoref{pr:kolmo-cen-tight} imply the 
corresponding conditions of \autoref{thm:tight}. 

Returning to our main task, once proved that the collection of measures $ \{\mu_{u,v}\}_{u > v}$ satisfies the conditions of \autoref{thm:weak-conv} and \autoref{thm:tight}, by exploiting Proposition \autoref{pr:kolmo-cen-tight}, we are able to assess the existence of some process whose finite dimensional distributions 
coincide with those of the weak limit $ \mu_v $ of $ \mu_{u,v}$ when $ u \downarrow v $. This measure will coincide with that induced by
$ B^\mu | \Lambda_v $ where 
\[
\Lambda_v = \Big \{\omega : \inf_{ 0 < z < T } \omega(z) > v , \omega(0)  = v \Big \}.
\]

This means that the continuous mapping theorem holds, i.e.
 $ \mu_{u,v} \circ g^{-1}$  \raisebox{-3pt} {$\xRightarrow{u\to v }$ }$  \mu_{v} \circ g^{-1} $ for any bounded uniformly continuous $ g $. 
 Using this fact we can derive the distribution of the maximum of $ \max B^\mu | \Lambda_v $ as the weak limit 
 of $ \max B^\mu | \Lambda_{u,v} $. 
 This was done in the driftless case by \citet{iglehart77}.

\section{Weak convergence to the Brownian meander with drift}\label{sec:joint}
The drifted Brownian meander can be viewed as a Brownian motion $ B^\mu $
with drift restricted on the subsets of continuous functions 
$ \Lambda_{u,v} = \{  \omega \in C[0,T]:  \min_{ 0\leq z \leq T} \omega(z) > v, B^\mu(0) = u \} $, with $ u>v $.
The probability measure we consider here is applied to the events 
$ A = \big \{ \bigcap_{j=1}^n \{ B^{\mu} (s_j) \in \mathrm d y_j \} \big \} $ under the 
condition $ \{ \min_{ 0\leq z \leq T} B^\mu (z) > v, B^\mu(0)=u \} $. 
For different values of $ v $ and of the starting point $ u $ we have sequences of 
processes for which we study the distributional structure. 

In order to write down explicitly  the conditional  probability of the event $ A $ we
need the distribution of an absorbing Brownian motion travelling on the set 
$ (v, \infty) $ starting at point $ u>v $ and with an absorbing 
barrier placed at $ v $ that is
\begin{align}\label{eq:bm-drift-recall-new}
&P\left\{  B^\mu(s) \in \mathrm d y\, ,  \min_{0\leq z \leq s} B^\mu(z)> v \Big \vert   B^\mu(0) = u\right \} 
=\\
&=\left( 
\frac{
	e^{- \frac{ (y- u - \mu s)^2}{2 s } } 
}{ %
	\sqrt{2\pi s }} - 
e^{ - 2 \mu (u - v) } 
\frac{
	e^{- \frac{ (y + u - 2v - \mu s)^2}{2 s } } 
}{ %
	\sqrt{2\pi s } }\right)	\, \mathrm d y \notag \\
&=
\Big( 
	e^{- \frac{ (y- u )^2}{2 s } } -
	e^{- \frac{ (2v - y - u)^2}{2 s } }
\Big)
	e^{ - \frac{\mu^2 s}{2} + \mu ( y-u) }
  \frac{ \mathrm d y}{\sqrt{2 \pi s }} \qquad y>v, u>v, s>0 \,\, .\notag 
\end{align} 
For $ \mu >0 $, the probability of the Brownian motion of never being captured 
by the absorbing barrier is equal to
\[ \lim_{s \to \infty } 
	P\Big\{ \min_{0\leq z \leq s} B^\mu(z)> v 
	\Big \vert   B^\mu(0) = u \big\} = 
	1 - e^{-2 \mu (u-v)} 	
 \]
and increases as the starting point goes further from the barrier. 
\begin{theorem}\label{thm:pre-mdr-joint}
	In view of \eqref{eq:bm-drift-recall-new} we have the following joint distributions
	\begin{align*}
	&P \bigg\{ \bigcap_{j=1}^n \left( B^\mu(s_j) \in \mathrm d y_j \right) \,\Big \vert \min_{0\leq z \leq t} B^\mu(z)> v , B^\mu(0)=u \bigg \} = \numberthis \label{eq:bmd-joint-dist-drift-new} \\
	&=
	\prod_{j=1}^{n} \Bigg[
	\frac{
		e^{- \frac{ (y_j - y_{j-1} - \mu( s_j - s_{j-1} ) )^2}{2 (s_j - s_{j-1}) } } 
	}{ %
		\sqrt{2\pi (s_j - s_{j-1}) } } - 
	e^{ - 2 \mu (y_{j-1} - v)} 
	\frac{
		e^{- \frac{ (y_j + y_{j-1} - 2v - \mu( s_j - s_{j-1} ) )^2}{2 (s_j - s_{j-1}) } } 
	}{ %
		\sqrt{2\pi (s_j - s_{j-1}) } }	\Bigg] \mathrm d y_j \quad \times \\
	& \qquad \times \quad 
	\frac{\displaystyle 
		P \Big\{   \min_{s_{n}\leq z \leq t} B^\mu(z)> v  \, \Big \vert B^\mu (s_n) = y_n 
		  \Big \} 
	}{
		\displaystyle
		P \Big \{  \min_{0\leq z \leq t} B^{\mu}(z)> v \,  \Big \vert  B^\mu(0) = u  \Big\}
	}  
	\\
	&=
	\prod_{j=1}^{n} \Bigg[
	\frac{
		e^{- \frac{ (y_j - y_{j-1} )^2}{2 (s_j - s_{j-1}) } } 
	}{ %
		\sqrt{2\pi (s_j - s_{j-1}) } } - 
	\frac{
		e^{- \frac{ (2v - y_j - y_{j-1} )^2}{2 (s_j - s_{j-1}) } } 
	}{ %
		\sqrt{2\pi (s_j - s_{j-1}) } }	\Bigg] \mathrm d y_j 
	\frac{ 
		\displaystyle
		\int_v^\infty 
		\bigg( 
		\frac{ 
			e^{- \frac{ (w - y_{n} )^2}{2 (t - s_{n}) } } - 
			e^{- \frac{ (2v - w - y_{n} )^2}{2 (t - s_{n}) } } 
		} 
		{ 
			 \sqrt{2\pi (t - s_{n}) }
		} 
		\bigg)  e^{\mu w} \,\mathrm d w 
		}{
		\displaystyle
		\int_v^\infty 
		\bigg(
		\frac{ 
		e^{- \frac{ (w - u )^2}{2 t } } - 
		e^{- \frac{ (2v - w - u )^2}{2 t } } 
	}{
		  \sqrt{2\pi t } 
}
		\bigg)  e^{\mu w} \,\mathrm d w  
	}  
	\end{align*}
	with $ y_i > v, i=1, \ldots,n $,  $ y_0=
	u $, $ 0<s_0 < s_1 < \ldots < s_j < \ldots < s_n < t\,,\,\, u>v $.
\end{theorem}
\begin{proof}
	
	\begin{align*}
	&
		P \bigg\{ \bigcap_{j=1}^n \left( B^\mu(s_j) \in \mathrm d y_j \right) \,\Big \vert \min_{0\leq z \leq t} B^\mu(z)> v , B^\mu(0)=u \bigg \}
		\numberthis \label{eq:joint-markov-a}
		\\
		&
		= 
		P_u \bigg\{ \bigcap_{j=1}^n \left( B^\mu(s_j) \in \mathrm d y_j  ,  \min_{s_{j-1} < z \leq s_j} B^\mu(z)> v \right) 
		\bigg \}
		P_u \bigg\{
		 \min_{0\leq z \leq t} B^\mu(z)> v \bigg \} ^{-1
		 }
	\end{align*}
	Considering the numerator of \eqref{eq:joint-markov-a} we
	have that
	\begin{align*}
	&P_u \bigg\{ \bigcap_{j=1}^n \left( B^\mu(s_j) \in \mathrm d y_j  ,  \min_{s_{j-1} < z \leq s_j} B^\mu(z)> v \right) 
	\bigg \} \numberthis \label{eq:joint-markov-b}
	\\
	&= 
	\mathbb E \left\{
	P_u \bigg\{ \bigcap_{j=1}^n \left( B^\mu(s_j) \in \mathrm d y_j  ,  \min_{s_{j-1} < z \leq s_j} B^\mu(z)> v \right) 
	\bigg | \mathscr F_{s_{n-1}}
	\bigg \} 
	\right\}
	\\
	&=
	\mathbb E \bigg \{
	P_u 
	\bigg\{ \bigcap_{j=1}^{n-1} \left( B^\mu(s_j) \in \mathrm d y_j  ,  \min_{s_{j-1} < z \leq s_j} B^\mu(z)> v \right) 
	\bigg \} \\
	& \qquad \times 
	P_u \bigg\{   B^\mu(s_{n}) \in \mathrm d y_n  ,  \min_{s_{n-1} < z \leq s_n} B^\mu(z)> v  
	\, \bigg | \mathscr F_{s_{n-1}}
	\bigg \} 
	\bigg\}
	\end{align*}
	
	Markovianity and result \eqref{eq:bm-drift-recall-new} applied successively yield 
	\eqref{eq:bmd-joint-dist-drift-new}. 
\end{proof}

\begin{corollary}
	For $ n=1 $, we obtain from \eqref{eq:bmd-joint-dist-drift-new} the one dimensional 
	distribution 
	\begin{align*}
		&P \Big\{  
		B^{\mu}  (s) \in \mathrm d y \, \Big | \min_{0\leq z \leq t} B^\mu(z)> v , B^\mu(0)=u 
		 \Big\}  \\
		&=
		 \Bigg[
		\frac{
			e^{- \frac{ (y - u - \mu s )^2}{2 s } } 
		}{ %
			\sqrt{2\pi s } } - 
		e^{ - 2 \mu (u - v)} 
		\frac{
			e^{- \frac{ (y + u - 2v - \mu s )^2}{2 s } } 
		}{ %
			\sqrt{2\pi s } }	\Bigg]   \cdot  
		\frac{\displaystyle 
			P \Big\{   \min_{s\leq z \leq t} B^\mu(z)> v  \, \Big \vert B^\mu (s) = y 
			\Big \} 
		}{
			\displaystyle
			P \Big \{  \min_{0\leq z \leq t} B^{\mu}(z)> v \,  \Big \vert  B^{\mu}(0) = u  \Big\}
		}  \mathrm d y
	  \numberthis \label{eq:bmd-onedim-dist-drift-new}
	\end{align*}
	for $ y>v, \, u>v $.
\end{corollary}
We now study the weak convergence of $ B^\mu | \Lambda_{u,v} $ as $ u \to v $. 


%
%
%

\begin{lemma}\label{lem:tight-cond-kolmo}
	For the sequence of conditional processes $ \{ B^\mu | \Lambda_{u,v} , u >v \} $ it holds that
	\begin{equation}\label{eq:lem-kolmo-stat}
	\lim_{ u \downarrow v}
	\mathbb E \left [ |B^\mu(t) - B^\mu(s)|^\alpha \Big| \min_{ 0 \leq z \leq T } B^\mu(z) > v , B^\mu(0)  = u \right]  \leq C |t-s|^{\alpha/2} \qquad  0 \leq s \leq t \leq T.
	\end{equation}
	\begin{proof}
		Having in mind the expression of the joint distribution given in \autoref{thm:pre-mdr-joint}, 
		the expectation in \eqref{eq:lem-kolmo-stat} can be written down as 
		\begin{align*}
			&\mathbb E \left [ |B^\mu(t) - B^\mu(s)|^\alpha \Big| \min_{ 0 \leq z \leq T } B^\mu(z) > v , B^\mu(0)  = u \right] 
			\\
			&= 
			\int_v^\infty \int_v^\infty |x-y|^\alpha 
			\left( 
				e^{- \frac{ (x - u )^2}{2s } } 
				 -
				e^{- \frac{ (2v - x - u )^2}{2s } } 
			\right)		
			\frac{ \mathrm d x}{	\sqrt{2\pi s } }  
			\left( 
			e^{- \frac{ (y - x )^2}{2(t-s) } } 
			-
			e^{- \frac{ (2v - y - x )^2}{2(t-s) } } 
			\right)		
			\frac{ \mathrm d y}{	\sqrt{2\pi (t-s) } }  \times 
			\\
			& \qquad \times 
			\frac{ 
				\displaystyle
				\int_v^\infty 
				\bigg( 
				\frac{ 
					e^{- \frac{ (w - y )^2}{2 (T - t) } } - 
					e^{- \frac{ (2v - w - y )^2}{2 (T - t) } } 
				} 
				{ 
					\sqrt{2\pi (T - t) }
				} 
				\bigg)  e^{\mu w} \,\mathrm d w 
			}{
				\displaystyle
				\int_v^\infty 
				\bigg(
				\frac{ 
					e^{- \frac{ (w - u )^2}{2T } } - 
					e^{- \frac{ (2v - w - u )^2}{2 T } } 
				}{
					\sqrt{2\pi T} 
				}
				\bigg)  e^{\mu w} \,\mathrm d w  
			}  .
		\end{align*}
		By taking the limit for $ u\to v  $ and,  by a dominated convergence argument, pushing the limit under the integral sign
		we have that 
		\begin{align*}
		&
		\lim_{ u \downarrow v}
		\mathbb E \left [ |B^\mu(t) - B^\mu(s)|^\alpha \Big| \min_{ 0 \leq z \leq T } B^\mu(z) > v , B^\mu(0)  = u \right]
		\numberthis 
		\\
		&= 
		\int_v^\infty \int_v^\infty |x-y|^\alpha 
		\frac{x-v}{s} 
		e^{- \frac{ (x - v )^2}{2s } } 
		\frac{ \mathrm d x}{	\sqrt{2\pi s } }  
		\left( 
		e^{- \frac{ (y - x )^2}{2(t-s) } } 
		-
		e^{- \frac{ (2v - y - x )^2}{2(t-s) } } 
		\right)		
		\frac{ \mathrm d y}{	\sqrt{2\pi (t-s) } }  \times 
		\\
		& \qquad \times 
		\frac{ 
			\displaystyle
			\int_v^\infty 
			\bigg( 
			\frac{ 
				e^{- \frac{ (w - y )^2}{2 (T - t) } } - 
				e^{- \frac{ (2v - w - y )^2}{2 (T - t) } } 
			} 
			{ 
				\sqrt{2\pi (T - t) }
			} 
			\bigg)  e^{\mu w} \,\mathrm d w 
		}{
			\displaystyle
			\int_v^\infty 
			\frac{w - v}{T}
			e^{- \frac{ (w - v )^2}{2T } }
			\frac{ e^{\mu w} }{
				\sqrt{2\pi T} }
			  \,\mathrm d w  
		}    
	\\
	&= C_{\mu, T}
	\int_0^\infty \int_0^\infty |x-y|^\alpha 
	\frac{x}{s} 
	e^{- \frac{x^2}{2s } } 
	\frac{ \mathrm d x}{	\sqrt{2\pi s } }  
	\left( 
	e^{- \frac{ (y - x )^2}{2(t-s) } } 
	-
	e^{- \frac{ (x + y )^2}{2(t-s) } } 
	\right)		
	\frac{ \mathrm d y}{	\sqrt{2\pi (t-s) } }  \times 
	\\
	& \qquad \times 
		\displaystyle
		\int_0^\infty 
		\bigg( 
		\frac{ 
			e^{- \frac{ (w - y )^2}{2 (T - t) } } - 
			e^{- \frac{ ( w + y )^2}{2 (T - t) } } 
		} 
		{ 
			\sqrt{2\pi (T - t) }
		} 
		\bigg)  e^{\mu w} \,\mathrm d w 
	\\
	&\leq
	C'_{\mu, T}
	\int_0^\infty \int_0^\infty |x-y|^\alpha 
	\frac{x}{s} 
	e^{- \frac{x^2}{2s } } 
	\frac{ \mathrm d x}{	\sqrt{2\pi s } }  
	e^{- \frac{ (y - x )^2}{2(t-s) } } 
	\frac{ \mathrm d y}{	\sqrt{2\pi (t-s) } }   
	\\
	\\
	&\leq
	C'_{\mu, T}
	\int_{0}^\infty \int_{-\infty}^\infty |x-y|^\alpha 
	\frac{x}{s} 
	e^{- \frac{x^2}{2s } } 
	\frac{ \mathrm d x}{	\sqrt{2\pi s } }  
	e^{- \frac{ (y - x )^2}{2(t-s) } } 
	\frac{ \mathrm d y}{	\sqrt{2\pi (t-s) } }  
	\\
	&=
	C'_{\mu, T}  \int_{0}^\infty 	\frac{x}{s} 
	e^{- \frac{x^2}{2s } }  \frac{ \mathrm d x}{	\sqrt{2\pi s } }  
	\int_{-\infty}^\infty |y|^\alpha 
	e^{- \frac{ y^2}{2(t-s) } } 
	\frac{ \mathrm d y}{	\sqrt{2\pi (t-s) } }  
	\\
	& \leq 
	C''_{\mu, T} |t-s|^{\alpha / 2}\, .
		\end{align*}
	\end{proof}
\end{lemma}

\begin{remark}
	This result shows that, as a consequence of the Kolmogorov-\v Centsov theorem, the drifted Brownian meander, i.e. the weak limit of the sequence $ \{ B^\mu | \Lambda_{u,v} , u >v \} $, is \emph{a.s.} locally H\"older continuous of parameter $\gamma \in (0, 1/2)$. The Brownian meander thus inherits the same H\"older exponent as
	the underlying  unconditional Brownian motion, as could be expected. 
\end{remark}

We are now able to prove the tightness of the family of measures induced by the drifted 
Brownian meander. 

\begin{theorem}\label{thm:tight-bmd}
	The family of measures $ \{\mu_{u,v} \}_u $ on $ (C[0,T], \mathscr C) $ indexed by $ u$, given by
	\[
	\mu_{u,v}(A) = P\Big(B^\mu \in A \Big | \min_{0\leq z\leq T} B^\mu(z) > v, B^\mu(0) = u\Big ) \qquad A \in \mathscr C, \quad u>v
	\]
	is tight. 
\end{theorem}
\begin{proof}
	The result follows as an application of Proposition \autoref{pr:kolmo-cen-tight}. 
	The proof that the  sequence of measures $ \{ \mu_{u,v}\}_{u > v} $ satisfies condition 
	$ (ii) $ is given in 
	\autoref{lem:tight-cond-kolmo}. Condition $ (i) $ clearly holds since
	$ P(B^\mu(0) = u) = 1 $ for $ u > v  $. 
	
\end{proof}

\begin{lemma}\label{lem:fin-dim}
	The finite dimensional distributions of $ B^\mu | \Lambda_{u,v} $ converge 
	to those of $ B^\mu | \Lambda_{v} = B^\mu| \{\inf_{ 0 < z < T } B^\mu(z) > v, B^\mu(0) = v \} $
	as $ u \downarrow v $. 
\end{lemma}
\begin{proof}
	As a consequence of Markovianity the terms for $ j=2, 3, \ldots n $ in the product in formula \eqref{eq:bmd-joint-dist-drift-new} do not depend on the 
	starting point $ u $. Thus it suffices to  compute the following limit
	\begin{align*}
		&
		P \bigg\{  
		B^\mu ( s) \in \mathrm d y \bigg |  \inf_{ 0 < z < t } B^\mu( z)  > v, B^\mu(0) = v  
		\bigg \} 
		\numberthis \label{eq:bmd-drift-dist-prooved}
		\\
		&=
		\lim_{ u \downarrow v} 
		P \bigg\{  
		B^\mu ( s) \in \mathrm d y \bigg |  \min_{ 0 \leq z \leq t } B^\mu( z)  > v, B^\mu(0) = u  
		\bigg \} 
		\\
		&=
		\frac{
			\frac{y - v}{\sqrt{2 \pi s^3} } e ^ {  -  \frac{(y - v - \mu s)^2}{2s} }
			P \Big\{  \min_{ s \leq z \leq t} B^\mu(z) > v \Big | B^\mu(s) = y  \Big\}      	
		}{
			\int_v^\infty \frac{w - v}{t}
			\frac{
			e^{  - \frac{ (w-y)^2}{2t} }	
		}{\sqrt{2 \pi t}}
		e^{  - \frac{\mu^2 t}{2} + \mu(w -v)  } 
		\, \mathrm d w
 		}
 	\,\mathrm d y
 	\\
	&=
\mathrm d y \left(\frac ts \right)^{\frac 32} 
\frac{
	(y - v) e^{  -\frac{(y-v)^2}{2s}  }
}{
	\int_0^{\infty} w  e^{  - \frac { w^2 } { 2t } + \mu w } \,\mathrm d w 
}
\int_0^\infty  \left( e^{ - \frac{(w-(y-v) )^2}{2 ( t-s)}}  - e^{ - \frac{(w + (y - v ) )^2}{2 ( t-s)}}  \right)
\frac{e^{\mu w }}{\sqrt{2 \pi( t-s)}} \, \mathrm d w  .
	\end{align*}

Consequently the finite dimensional distributions of the limiting process can be written 
down explicitly as
\begin{align*}
&P \bigg\{ \bigcap_{j=1}^n \left( B^\mu(s_j) \in \mathrm d y_j \right) \,\Big \vert \inf_{0< z < t} B^\mu(z)> v , B^\mu(0)=v \bigg \}	
\\
&=
\mathrm d y_1 \left(\frac t{s_1} \right)^{\frac 32} 
\frac{
	(y_1 - v) e^{  -\frac{(y_1-v)^2}{2s_1}  }
}{
	\int_0^{\infty} w  e^{  - \frac { w^2 } { 2t } + \mu w }  \,\mathrm d w
}
\prod_{j=2}^{n} \Bigg[
\frac{
	e^{- \frac{ (y_j - y_{j-1} )^2}{2 (s_j - s_{j-1}) } } 
}{ %
	\sqrt{2\pi (s_j - s_{j-1}) } } - 
\frac{
	e^{- \frac{ (2v - y_j - y_{j-1} )^2}{2 (s_j - s_{j-1}) } } 
}{ %
	\sqrt{2\pi (s_j - s_{j-1}) } }	\Bigg] \mathrm d y_j 
\\
& \qquad \times 
\int_0^\infty  \left( e^{ - \frac{(w-(y_n-v) )^2}{2 ( t-s_n)}}  - e^{ - \frac{(w + (y_n - v ) )^2}{2 ( t-s_n)}}  \right)
\frac{e^{\mu w }}{\sqrt{2 \pi( t-s_n)}} \, \mathrm d w  . 
\end{align*}

with $ y_i > v, i=1, \ldots,n $,  $ 0< s_1 < \ldots < s_j < \ldots < s_n < t $.

\end{proof}

\begin{theorem}\label{thm:weak-conv-mdr}
	The following weak limit holds:
	\begin{equation}\label{eq:weak-conv-mdr}
	B^\mu \Big |\Big \{  \min_{ 0 \leq z \leq t } B^\mu > v, B^\mu(0) = u \Big \} \xRightarrow[u\downarrow v]{}
	B^\mu \Big | \Big \{ \inf_{ 0 < z < t } B^\mu > v, B^\mu(0) = v \Big \}
	\end{equation}
\end{theorem}
\begin{proof}
	In view of \autoref{lem:fin-dim} and \autoref{thm:tight-bmd}, result \eqref{eq:weak-conv-mdr} is an application of \autoref{thm:weak-conv}. 
\end{proof}

We now check that the maximum variation around the starting point for the drifted meander 
tends to zero, in a sense which is made precise in the following corollary.

\begin{corollary} \label{lem:tight-cond-calc}
	The following limit holds 
	\begin{equation}\label{eq:lemma-cond}
	\lim_{\delta \to 0}
	\lim_{u \to v} 
	P \Big( \max_{0\leq z \leq \delta } | B^\mu ( z)  - B^\mu ( 0 ) | < \eta  \, \Big | 
	\min_{ 0\leq z \leq t} B^\mu ( z) > v ,  B^\mu(0) = u \Big) = 1
	\end{equation}
	for all $ \eta >0 $.
\end{corollary}
\begin{proof}
	The probability appearing in  \eqref{eq:lemma-cond} can be developed as 
	\begin{align*}
	&
	F(\delta, v, u, \eta; t)  \\
	&=
	\int_{u-\eta}^{u + \eta} 
	\frac{\displaystyle 
		P \Big\{   \min_{\delta \leq z \leq t} B^\mu(z)> v  \, \Big \vert B^\mu (\delta) = y 
		\Big \} 
	}{
		\displaystyle
		P \Big \{  \min_{0\leq z \leq t} B^{\mu}(z)> v \,  \Big \vert  B^{\mu}(0) = u  \Big\}
	}  \times 
	\\
	& 
	\quad \times P \Big( u - \eta < \min_{ 0\leq z \leq \delta} B^\mu ( z) < 
	\max_{0\leq z \leq \delta}  B^\mu (z )  < u + \eta \,\, , 
	\min_{ 0\leq z \leq \delta} B^\mu (z) > v ,  
	B^\mu(\delta) \in \mathrm d y \, \Big| B^\mu(0) = u \Big )
	\numberthis \label{eq:lemma-prob-1}
	\,\, .
	\end{align*}
	For $ u $ decreasing to $ v $ we have that $ u - \eta < v $ and this implies that 
	$ \Big( \min_{ 0\leq z \leq \delta} B(z) > v  \Big ) \subset   \Big( \min_{ 0\leq z \leq \delta} B(z) > u - \eta \Big ) $
	and \eqref{eq:lemma-prob-1} simplifies as 
	\begin{align*}
	F(\delta, u, v, \eta;t) &= 
	\int_{v}^{u + \eta} P\left\{ v < \min_{ 0\leq z \leq \delta} B^\mu ( z )  
	< \max_{0 \leq z <\delta } B^\mu(z) < u + \eta,
	B^\mu(\delta ) \in \mathrm d y | B^\mu(0) = u  \right\} \times \\
	& \qquad \times 
	\frac{
		P\left\{ \min_{ \delta \leq z \leq t} B^\mu( z) > v \Big | B^\mu(\delta ) = y \right\}
	}{
		P\left\{ \min_{ 0 \leq z \leq t} B^\mu( z) > v \Big | B^\mu(0 ) = u \right\}
	}
	\numberthis  \label{eq:lemma-prob-simp}
	\\
	&=
	\int_{v}^{u + \eta } \frac{\mathrm d y}{\sqrt{2 \pi \delta }}
	\sum_{k = - \infty}^{+ \infty}
	\left\{
	e^{  -  \frac{ (y - u -2k(\eta + u - v))^2}{2 \delta}  } 
	- 
	e^{  -  \frac{ (2 v - y - u +2k(\eta + u - v))^2}{2 \delta}  }
	\right\}
	\times 
	\\
	&\qquad \times 
	e^{ - \frac{\mu^2 \delta}{2}  + \mu ( y - u  - 2k(u + \eta - v)) }
	\times \\
	& \qquad \times 
	\frac{
		P\left\{ \min_{ \delta \leq z \leq t} B^\mu( z) > v \Big | B^\mu(\delta ) = y \right\}
	}{
		P\left\{ \min_{ 0 \leq z \leq t} B^\mu( z) > v \Big | B^\mu(0 ) = u \right\} 
	} \,\, .
	\end{align*}
	Since the probability
	\begin{align*}
	P\left\{ \min_{ 0 \leq z \leq t} B^\mu( z) > v \Big | B^\mu(0 ) = u \right\} &=
	\int_v^\infty 
	P\left\{ 
	\min_{ 0 \leq z \leq t} B^\mu( z) > v, B^\mu(t) \in \mathrm d y \Big | B^\mu(0) = u
	\right\} \\
	&=
	\int_{v}^{\infty} 
	\frac{
		e^{ - \frac{ (y - u )^2}  { 2t}}  - e^{ - \frac{(y+u - 2v)^2}{2t} }  
	}{\sqrt{2 \pi t}}
	e^{ - \frac{\mu^2t} {2  } + \mu ( y - u) }  \mathrm d y 
	\end{align*}
	converges to zero as $  u \to v $, we can apply De l'H\^opital rule to \eqref{eq:lemma-prob-simp} for the limit 
	\begin{align*}
	& \lim\limits_{u \downarrow v} F( \delta, v, u, \eta ; t)  
	\numberthis 
	\label{eq:lemma-prob-hop}
	\\
	&=
	\frac{ 
		\int_{v}^{v+\eta} 
		\sum_{k= - \infty}^{+ \infty} 
		\frac{
			e^{  - \frac{  ( y - v - 2k\eta  )^2}{2 \delta}  }
		}{\sqrt{2 \pi \delta }}   \, \frac{(y - v - 2 k \eta )}{\delta} 
		e^{   
			- \frac{\mu^2 \delta }{2} + \mu ( y - v - 2 k \eta) 
		}
		P\left\{ \min_{ \delta \leq z \leq t} B^\mu( z) > v \Big | B^\mu(\delta ) = y \right\}
		\, \mathrm d y 
	}
	{
		\int_v^\infty  (y-v) \frac {e^{ -   \frac{  ( y - v)^2}{2t} }}{t \sqrt { 2 \pi t} }
		e^{   
			- \frac{\mu^2 t}{2} + \mu ( y - v ) 
		}
		\mathrm d y 
	} \,\,.
	\end{align*}
	We treat separately the cases $ k=0 $ and $  k \neq 0 $. 
	The functions 
	\begin{equation}
	q(y, v + 2 k \eta, \delta) = \frac{(y - v-2 k \eta)}{\delta} 
	\frac{e^{ -  \frac{ (y-v - 2 k \eta)^2}{2 \delta }  }}{\sqrt{2 \pi \delta}}
	\end{equation}
	for $  \delta \to 0 $ converge to the Dirac delta functions with poles at $  y  = v  + 2 k \eta $
	which are outside the integration interval for all $  k \neq 0 $. 
	We now consider the term $ k = 0 $. 
	\begin{align*}
	&
	\int_{v }^{v + \eta } e^{ - \frac{(y - v - \mu \delta)^2}{2 \delta}}
	\frac{(y-v)}{\sqrt{2 \pi \delta^3}} \, \mathrm d y 
	\int_{v }^{\infty} 
	\frac{
		e^{  - \frac{(w-y)^2}{2 ( t - \delta) } } - e^{-  \frac{(w + y - 2v)^2}{2(t - \delta)}  }
	}{\sqrt{ 2 \pi ( t- \delta)}} e^{  - \frac{\mu^2 ( t - \delta)}{2} + \mu ( w - y)  } \, \mathrm d w
	\numberthis \label{eq:lemma-prob-eps}
	\\
	&=
	\int_{0}^\eta 
	\frac{
		y e^{    -   \frac{(y-\mu \delta )^2}{2 \delta} }
	}{\sqrt{2 \pi \delta ^ 3}} \,\mathrm d y 
	\int_{ v }^\infty 
	\frac{
		e^{ - \frac{(w - v - y)^2}{2 ( t - \delta)} } 
		- 
		e^{  - \frac{ ( w + v + y - 2 v )^2   }  { 2 (t-\delta)   } }
	}{\sqrt{2 \pi (t - \delta)}} 
	e^{  - \frac{\mu^2 ( t - \delta)}{2} + \mu ( w -v  - y)  } \, \mathrm d w
	\\
	&=
	\int_{0}^\eta 
	\frac{
		y e^{    -   \frac{(y-\mu \delta )^2}{2 \delta} }
	}{\sqrt{2 \pi \delta ^ 3}} \,\mathrm d y 
	\int_{0} ^ \infty 
	\frac{
		e^{ - \frac{(w - y)^2}{2 ( t - \delta)} } 
		- 
		e^{  - \frac{ ( w +  y   )^2   }  { 2 (t-\delta)   } }
	}{\sqrt{2 \pi (t - \delta)}} 
	e^{  - \frac{\mu^2 ( t - \delta)}{2} + \mu ( w - y)  } \, \mathrm d w
	\\
	&=
	\int_{0}^\eta 
	\frac{
		y e^{    -   \frac{y^2}{2 \delta} }
	}{\sqrt{2 \pi \delta ^ 3}} 
	e^{  - \frac{\mu^2  \delta}{2} + \mu y  }
	\,\mathrm d y 
	\int_{0} ^ \infty 
	\frac{
		e^{ - \frac{(w - y)^2}{2 ( t - \delta)} } 
		- 
		e^{  - \frac{ ( w +  y   )^2   }  { 2 (t-\delta)   } }
	}{\sqrt{2 \pi (t - \delta)}} 
	e^{  - \frac{\mu^2 ( t - \delta)}{2} + \mu ( w - y)  } \, \mathrm d w
	\\
	&=
	e^{  - \frac{\mu^2  t}{2}  }
	\int_{0}^\eta 
	\frac{
		y e^{    -   \frac{y^2}{2 \delta} }
	}{\sqrt{2 \pi \delta ^ 3}} 
	\,\mathrm d y 
	\int_{0} ^ \infty 
	\frac{
		e^{ - \frac{(w - y)^2}{2 ( t - \delta)} } 
		- 
		e^{  - \frac{ ( w +  y   )^2   }  { 2 (t-\delta)   } }
	}{\sqrt{2 \pi (t - \delta)}} 
	e^{   \mu w  } \, \mathrm d w \,\, 
	\\
	%
	%
	%
	%
	%
	%
	&=
	e^{ - \frac{\mu^2 t}{2}} 
	\int_0^\eta \frac{ y e^{ - \frac{ y^2}{2 \delta }}}{\sqrt{2 \pi \delta^3}} \, \mathrm d y
	\int_0^\infty 
	\frac{ e^{- \frac{ ( w - y)^2}{2(t-\delta )}} 
		- e^{   - \frac{ (w+y)^2}{2 ( t-\delta)}   } } {\sqrt{2 \pi (t - \delta)}}
	e^{ \mu w } \, \mathrm d w
	\numberthis \label{eq:lemma-prob-even}
	\\
	&=
	\frac{ 
		e^{ - \frac{\mu^2 t}{2}}
	}{2} 
	\int_{- \eta} ^ \eta 
	\frac{ y e^{ - \frac{ y^2}{2 \delta }}}{\sqrt{2 \pi \delta^3}} \, \mathrm d y
	\int_0^\infty 
	\frac{ e^{- \frac{ ( w - y)^2}{2(t-\delta )}} 
		- e^{   - \frac{ (w+y)^2}{2 ( t-\delta)}   } } {\sqrt{2 \pi (t - \delta)}}
	e^{ \mu w } \, \mathrm d w \,\, .
	\end{align*}
	Note that the function 
	\begin{equation*}
	g(y,w ; \delta )  = \frac{ y e^{ - \frac{ y^2}{2 \delta }}}{\sqrt{2 \pi \delta^3}} 
	\int_0^\infty 
	\frac{ e^{- \frac{ ( w - y)^2}{2(t-\delta )}} 
		- e^{   - \frac{ (w+y)^2}{2 ( t-\delta)}   } } {\sqrt{2 \pi (t - \delta)}}
	e^{ \mu w }\, \mathrm d w
	\end{equation*}
	is even with respect to the variable $ y $ and this justifies \eqref{eq:lemma-prob-even}.
	An integration by parts in \eqref{eq:lemma-prob-even} yields
	\begin{align*}
	&
	-\frac{ 
		e^{ - \frac{\mu^2 t}{2}}
	}{2} 
	\frac{ e^{ - \frac{ y^2}{2 \delta }}}{\sqrt{2 \pi \delta}} 
	\int_0^\infty 
	\frac{ e^{- \frac{ ( w - y)^2}{2(t-\delta )}} 
		- e^{   - \frac{ (w+y)^2}{2 ( t-\delta)}   } } {\sqrt{2 \pi (t - \delta)}}
	e^{ \mu w } \, \mathrm d w \Big| _{y = - \eta} ^ { y = \eta} 
	\\
	& \qquad 
	+
	\frac{ 
		e^{ - \frac{\mu^2 t}{2}}
	}{2} 
	\int_{- \eta} ^ \eta 
	\frac{ e^{ - \frac{ y^2}{2 \delta }}}{\sqrt{2 \pi \delta}} \, \mathrm d y
	\int_0^\infty 
	\bigg\{ 
	\frac{w - y }{t - \delta}
	e^{- \frac{ ( w - y)^2}{2(t-\delta )}}
	+ 
	\frac{w + y }{t - \delta}
	e^{   - \frac{ (w+y)^2}{2 ( t-\delta)}   } 
	\bigg \}
	\frac{ e^{ \mu w }} {\sqrt{2 \pi (t - \delta)}}
	\, \mathrm d w
	\end{align*}
	If we let $  \delta \to 0 $, for $  \eta>0 $, by denoting with $ \mu^\delta_x(\cdot) $ a Dirac point mass at $ x $,  we obtain
	\begin{align*}
	\numberthis 
	&  
	\frac{ 
		e^{ - \frac{\mu^2 t}{2}}
	}{2} 
	\int_{- \eta} ^ \eta 
	\mu^\delta_0(y) \, \mathrm d y
	\int_0^\infty 
	\bigg\{ 
	\frac{w - y }{t }
	e^{- \frac{ ( w - y)^2}{2 t}}
	+ 
	\frac{w + y }{t }
	e^{   - \frac{ (w+y)^2}{2  t }   } 
	\bigg \}
	\frac{ e^{ \mu w }} {\sqrt{2 \pi t }}
	\, \mathrm d w
	\\
	&=
	e^{ - \frac{\mu^2 t}{2}}
	\int_0^\infty 
	\frac{w  }{t }
	\frac{ e^{- \frac{ w^2}{2 t}  + \mu w}    } {\sqrt{2 \pi t }}
	\, \mathrm d w \,\, .
	\end{align*}
	This coincides with the denominator of \eqref{eq:lemma-prob-hop}
	and this concludes the proof of \cref{lem:tight-cond-calc}.
\end{proof}

\begin{corollary}
	It holds that
	\begin{align}
	& \lim_{\delta \to 0} \lim_{ u \downarrow v} P \left\{  \max_{s - \delta \leq z \leq s +\delta } |B^\mu (z) | \leq \eta \,\Big | \min_{ 0 \leq z \leq t } B^\mu(z) > v, B^\mu (0) = u\right\} 
	\\
	&= 
	P\left\{ B^\mu (s ) \leq \eta \Big  | \inf_{ 0 < z < t } B^\mu(z) > v, B^\mu(0) = v \right\} \, .
	\notag 
	\end{align}
	\begin{proof}
		\begin{align*}
		&
		P \left\{  \max_{s - \delta \leq z \leq s +\delta } |B^\mu (s) | \leq \eta \,\Big | \min_{ 0 \leq z \leq t } B^\mu(z) > v, B^\mu (0) = u\right\} 
		\numberthis \label{eq:cor-max-ini}
		\\
		&=
		\int_v^{\eta} \int_v^{\eta} 
		P \left\{
		\min_{ 0 \leq z \leq s - \delta  } B^\mu(z) > v, 
		B^\mu(s - \delta) \in \mathrm d w \Big | B^\mu(0) = u \right\}
		\\
		& \quad \times 
		P \left\{
		\min_{ s - \delta \leq z \leq s +  \delta  } B^\mu(z) > v, 
		-\eta < \max_{ s - \delta \leq z \leq s +  \delta  } B^\mu(z) <
		\eta, 
		B^\mu(s + \delta) \in \mathrm d y \Big | B^\mu(s - \delta) = w \right\}
		\\
		& \quad \times 
		P \left\{
		\min_{ s + \delta \leq z \leq t  } B^\mu(z) > v, 
		\Big | B^\mu(s + \delta ) = y \right\}
		\\
		& \quad \times 
		\left(
		\int_v^{\infty}
		P \left\{
		\min_{ 0 \leq z \leq t  } B^\mu(z) > v, 
		B^\mu(t) \in \mathrm d w \Big | B^\mu(0) = u \right\}
		\right)^ { - 1 }
		\end{align*}
		We first perform the limit for $ u \downarrow v $ which involves the following
		two terms of \eqref{eq:cor-max-ini}
		\begin{align*}
		& 
		\lim_{ u \downarrow v} 
		\frac{
			\int_v^{ \eta }
			P \left\{
			\min_{ 0 \leq z \leq s - \delta  } B^\mu(z) > v, 
			B^\mu(s - \delta) \in \mathrm d w \Big | B^\mu(0) = u \right\}
		}{
			\int_v^{\infty}
			P \left\{
			\min_{ 0 \leq z \leq t  } B^\mu(z) > v, 
			B^\mu(t) \in \mathrm d w \Big | B^\mu(0) = u \right\}
		}
		\\
		&=
		\lim_{ u \downarrow v} 
		\frac{
			\int_v^{ \eta }
			\Big( 
			e^{- \frac{ (w- u )^2}{2 (s - \delta )} } -
			e^{- \frac{ (2v - w - u)^2}{2 (s - \delta )} }
			\Big)
			e^{ - \frac{\mu^2 (s - \delta )}{2} + \mu ( w-u) }
			\frac{ \mathrm d w}{\sqrt{2 \pi (s - \delta ) }}
		}{
			\int_v^{\infty} 
			\Big( 
			e^{- \frac{ (w- u )^2}{2 t } } -
			e^{- \frac{ (2v - w - u)^2}{2 t } }
			\Big)
			e^{ - \frac{\mu^2 t}{2} + \mu ( w-u) }
			\frac{ \mathrm d w}{\sqrt{2 \pi t }}
		}
		\\
		&=
		\frac{
			\int_v^{\eta }
			\frac{w - v}{s - \delta }
			\frac{
				e^{  - \frac{ (w-v)^2}{2(s - \delta )} }	
			}{\sqrt{2 \pi t}}
			e^{  - \frac{\mu^2 (s - \delta )}{2} + \mu(w -v)  } 
			\, \mathrm d w
		}
		{
			\int_v^{\infty } 
			\frac{w - v}{t}
			\frac{
				e^{  - \frac{ (w-y)^2}{2t} }	
			}{\sqrt{2 \pi t}}
			e^{  - \frac{\mu^2 t}{2} + \mu(w -v)  } 
			\, \mathrm d w
		}
		\end{align*}
		In order to compute the limit for $ \delta \to 0  $ we first consider the 
		term
		\begin{align*}\label{key}
		&
		\int_v^{\eta }
		P \left\{
		\min_{ s - \delta \leq z \leq s +  \delta  } B^\mu(z) > v, 
		-\eta < \max_{ s - \delta \leq z \leq s +  \delta  } B^\mu(z) <
		\eta, 
		B^\mu(s + \delta) \in \mathrm d y \Big | B^\mu(s - \delta) = w \right\}
		\\
		&=
		\int_v^{\eta }
		P \left\{
		\min_{ 0 \leq z 2 \delta  } B^\mu(z) > v, 
		-\eta < \max_{0 \leq z \leq 2  \delta  } B^\mu(z) <
		\eta, 
		B^\mu(2 \delta) \in \mathrm d y \Big | B^\mu(0) = w \right\}
		\numberthis
		\\
		&=
		\int_v^{\eta }
		\frac{\mathrm d y}{\sqrt{2 \pi t}}
		\sum_{k = - \infty}^{+ \infty}
		\left\{
		e^{  -  \frac{ (y - w -2k(\eta + w - v))^2}{4 \delta}  } 
		- 
		e^{  -  \frac{ (2 v - y - u +2k(\eta + w - v))^2}{4 \delta}  }
		\right\}
		e^{ - \mu^2 \delta  + \mu ( y - u  - 2k(w + \eta - v)) }
		\end{align*}
		where we used the well-known joint distribution of the Brownian motion, 
		its maximum and its minimum (see for example \cite{borodin2012handbook}, p. 271).  When $ \delta  $ tends to zero all the terms in the last summation tend to 
		Dirac's point masses. They all have poles outside the interval of integration 
		except the first term for $ k = 0 $, which has a pole in $ y = w $. 
		Thus, in conclusion, we have that
		\begin{align*}
		& \lim_{\delta \to 0} \lim_{ u \downarrow v} P \left\{  \max_{s - \delta \leq z \leq s +\delta } |B^\mu (s) | \leq \eta \,\Big | \min_{ 0 \leq z \leq t } B^\mu(z) > v, B^\mu (0) = u\right\} 
		\\
		&=
		\int_v^{\eta }
		\frac{w - v}{s}
		\frac{
			e^{  - \frac{ (w-v)^2}{2s} }	
		}{\sqrt{2 \pi t}}
		e^{  - \frac{\mu^2 s}{2} + \mu(w -v)  } 
		P \left\{
		\min_{ s \leq z \leq t  } B^\mu(z) > v, 
		\Big | B^\mu(s ) = w \right\}
		\, \mathrm d w
		\\
		& \quad \times
		\left(
		\int_v^{\infty } 
		\frac{w - v}{t}
		\frac{
			e^{  - \frac{ (w-y)^2}{2t} }	
		}{\sqrt{2 \pi t}}
		e^{  - \frac{\mu^2 t}{2} + \mu(w -v)  } 
		\, \mathrm d w
		\right)^{\!\!-1}. \numberthis 
		\end{align*}
		The last expression corresponds to the univariate distribution of 
		the drifted Brownian meander, whose density is given in \eqref{eq:bmd-drift-dist-prooved}. 
	\end{proof}
\end{corollary}

\begin{remark}

Either from the last expression of \eqref{eq:bmd-drift-dist-prooved} or directly from the first 
one we get
\begin{align*}
	&
	P \bigg\{  
	B^\mu ( t) \in \mathrm d y \bigg |  \inf_{ 0 < z < t } B^\mu( z)  > v, B^\mu(0) = v  
	\bigg \} 
	\\
	&=
	\frac{
		(y - v )   e^{  - \frac{ (y-v)^2}{2t} } e^{\mu y } 
	}{
		\int_v^\infty (w-v) e^{  - \frac{ (w-v)^2}{2t}   }  e^{ \mu w} \, \mathrm d w
	} \,\mathrm d y
	\numberthis \label{eq:bmd-drift-dist-t-prooved} \quad , \qquad y > v\,\, .
\end{align*}

For $  \mu = 0  $ we retrieve from \eqref{eq:bmd-drift-dist-t-prooved} the truncated Rayleigh distribution 
\[
P \bigg\{  
B ( t) \in \mathrm d y \bigg |  \inf_{ 0 < z < t } B( z)  > v, B(0) = v  
\bigg \} = 
\frac{y - v}{t} e^{  - \frac{(y-v)^2}{2t}  } \,\mathrm d y\,\, , \qquad y > v \,\, .
\]
	
\end{remark}


%
\section{About the maximum of the Brownian meander}
In this section we study the distribution of the maximum of the drifted Brownian meander and then we use the results about the maximum to derive the distribution of the first passage time.
\subsection{Maximum.}
As far as the maximum is concerned we have that for $ s<t, x>v, x>u $, 
\begin{align*}
&
P\bigg\{ \max_{0 \leq z \leq s} B^\mu(z) <x \Big |  \min_{0 \leq z \leq t} B^\mu (z) > v, B^\mu(0) = u  \bigg\}
\numberthis 
\label{eq:max-bmd-ini}
\\
&=
\int_v^x 
P \bigg\{  v < \min_{0 \leq z \leq s} B^\mu(z) < \max_{0 \leq z \leq s} B^\mu(z) < x , 
	B^\mu(s) \in  \mathrm d y
	\Big | 
	B^\mu(0) = u
	 \bigg\} \times 
 	\\
 	& \qquad 
 	\times 
 	\frac{
 		P\Big\{   \min_{s \leq z \leq t} B^\mu(z)   > v  \Big | B^\mu ( s) = y \Big\}
 	}{
 		P\Big\{   \min_{0 \leq z \leq t} B^\mu(z)   > v  \Big | B^\mu ( 0) = u \Big\}
	}
	\\
	&=
	\int_v^x \, \frac{\mathrm d y}{\sqrt{2 \pi s}} \sum_{k = - \infty}^{+\infty}
	\left[ 
	e^{    -   \frac{(y - u - 2k(x-v))^2}{2s} } 
	- 
	e^{    -   \frac{(2v - y - u +2k(x-v))^2}{2s} }  
	\right ] 
	e^{ - \frac{\mu^2s}{2} + \mu ( y - u - 2k(x-v))} 
	\,\,\times 
	\\
	& \qquad \times 
	\frac{
		P\Big\{   \min_{s \leq z \leq t} B^\mu(z)   > v  \Big | B^\mu ( s) = y \Big\}
	}{
		\int_v^\infty 
		\left[ 
		e^{  - \frac{(y-u)^2}{2t}  } 
		- 
		e^{  - \frac{(2v - y-u)^2}{2t}  } 
		\right]
		e^{  - \frac{\mu^2 t}{2} + \mu ( y - u)   } \, \frac{ \mathrm d y}{\sqrt{2 \pi t} } 
	} \, \, .
\end{align*}
A certain simplification is obtained by letting $  s \to t $ and formula 
\eqref{eq:max-bmd-ini}  becomes
\begin{align*}
	&
	P\bigg\{ \max_{0 \leq z \leq t} B^\mu(z) <x \Big |  \min_{0 \leq z \leq t} B^\mu (z) > v, B^\mu(0) = u  \bigg\}
	\numberthis 
	\label{eq:max-bmd-t}
	\\
	&=
	\frac{
		\int_v^x \sum_{k = - \infty}^{+\infty}
		\left[ 
			e^{    -   \frac{(y - u - 2k(x-v))^2}{2t} } 
		- 
			e^{    -   \frac{(2v - y - u +2k(x-v))^2}{2t} }  %
		\right ] 
		e^{ - \frac{\mu^2t}{2} + \mu ( y - u - 2k(x-v))}
		\, \frac{ \mathrm d y}{\sqrt{2 \pi t} }
	}{
		\int_v^\infty 
		\left[ 
			e^{  - \frac{(y-u)^2}{2t}  }  
		- 
			e^{  - \frac{(2v - y-u)^2}{2t}  } 
		\right]
		e^{  - \frac{\mu^2 t}{2} + \mu ( y - u)   } \, \frac{ \mathrm d y}{\sqrt{2 \pi t} }
	} \, \, .
\end{align*}
Formula \eqref{eq:max-bmd-ini} for $  u \downarrow v $ yields
\begin{align*}
	&
	P\bigg\{ \max_{0 \leq z \leq s} B^\mu(z)<x \Big |  \inf_{0 < z < t} B^\mu (z) > v, B^\mu(0) = v  \bigg\}
	\numberthis 
	\label{eq:max-bmd-v}
	\\
	&=
		\int_v^x \sum_{k = - \infty}^{+\infty}
			e^{    -   \frac{(y - v - 2k(x-v))^2}{2 s} } 
		\frac { (y - v - 2k (x-v) )}{s\sqrt{2 \pi s}}
		e^{ - \frac{\mu^2 s}{2} + \mu ( y - v- 2k(x-v))}
		\, \times 
		\\
		& \qquad \times 
		\frac{ 
		P\Big\{   \min_{s \leq z \leq t} B^\mu(z)   > v  \Big | B^\mu ( s) = y \Big\}
		\, 	\mathrm d  y
		}{
		\int_v^\infty 
			e^{  - \frac{(w-v)^2}{2t}  } 
		\frac{ (w-v) }{ t \sqrt{2 \pi t } }
		e^{  - \frac{\mu^2 t}{2} + \mu ( w - v)   } \,  \mathrm d  w
	} 
	\\
	&=
	\int_v^x \sum_{k = - \infty}^{+\infty}
		e^{    -   \frac{(y - v - 2k(x-v))^2}{2 s} } 
	\frac { (y - v - 2k (x-v) )}{s\sqrt{2 \pi s }}
	e^{ -2 k \mu (x-v)}
	\, \times 
	\\
	& \qquad \times 
	\frac{ 
		\int_v^\infty
		\left[
			e^{  - \frac{(w - y )^2}{  2 (t-s) }    }
			- 
			e^{  - \frac{  (w + y - 2 v)^2}{ 2 (t-s) }    } 
			\right]
	e^{ \mu w} \, \frac{ \mathrm d w}{\sqrt{2 \pi (t-s)} }
		\,\mathrm d y
	}{
		\int_v^\infty 
			e^{  - \frac{(w-v)^2}{2t}  } 
		\frac { (w-v) }{t \sqrt{2 \pi t } } 
		e^{ \mu w   } \,  \mathrm d w
	} \, .
\end{align*}
Formula \eqref{eq:max-bmd-t} for $  u \downarrow v $ leads to 
\begin{align*}
		&
		P\bigg\{ \max_{0 \leq z \leq t} B^\mu(z) <x \Big |  \inf_{0 < z < t} B^\mu (z) > v, B^\mu(0) = v  \bigg\}
		\numberthis 
		\label{eq:max-bmd-t-v}
		\\
		&=
		\frac{  
		\sum_{k = -\infty} ^ { + \infty}
		e^{  - 2 \mu k (x-v)  }
		\int_v^x
			e^ {   -  \frac{(y - v - 2k(x-v))^2}{2t}  }
			(y - v - 2k ( x-v) )
			e^{ \mu y } \, \mathrm d y
		}{
		\int_v^\infty 
		e^{  - \frac{(y - v)^2}{2t}  } (y-v) 
		e^{ \mu y } \,\mathrm d y	
		      }
	    \\
	    &=
	    \frac{  
	    	\sum_{k = -\infty} ^ { + \infty}
	    	e^{  - 2 \mu k (x-v)  }
	    	\int_0^{x-v}
	    	e^ {   -  \frac{(w  - 2k(x-v))^2}{2t}  }
	    	(w- 2k ( x-v) )
	    	e^{ \mu w } \, \mathrm d w
	    }{
	    	t[1 + \mu \int_0^\infty 
	    	e^{  - \frac{y^2}{2t}  } 
	    	e^{ \mu y } \,\mathrm d y]	
	    } \, .
\end{align*}
Observe that the passage to the limit fo $ u \downarrow v $ is justified by \autoref*{thm:weak-conv-mdr} and by the continuous mapping theorem as observed
in Section 2.  For $ v=0 $ result \eqref{eq:max-bmd-t-v} further simplifies as 
\begin{align*}
&
P\bigg\{ \max_{0 \leq z \leq t} B^\mu(z) <x \Big |  \inf_{0 < z < t} B^\mu (z) > 0, B^\mu(0) = 0  \bigg\}
\numberthis 
\label{eq:max-bmd-t-v0}
\\
&=
\frac{  
	\sum_{k = -\infty} ^ { + \infty}
	e^{  - 2 \mu k x  }
	\int_0^{x}
	e^ {   -  \frac{(w  - 2kx)^2}{2t}  }
	(w- 2kx )
	e^{ \mu w } \, \mathrm d w
}{
	t[ 1 + \mu \int_0^\infty 
	e^{  - \frac{y^2}{2t}  } 
	e^{ \mu y } \,\mathrm d y	]
}
\\
&=
\frac{  
	\sum_{r = -\infty} ^ { + \infty}
	(-1)^r
	e^{  -  \mu r x - \frac{x^2 r^2}{2t}  } 
	+ \mu
	\sum_{r = -\infty}^{+ \infty}
	\int_{- 2 rx}^{x - 2 r x }
	e^ {   -  \frac{w ^2}{2t}  + \mu w}
 \, \mathrm d w
}{
	1 + \mu \int_0^\infty 
	e^{  - \frac{y^2}{2t}  } 
	e^{ \mu y } \,\mathrm d y	
} \, \, .
\end{align*}
Formula \eqref{eq:max-bmd-t-v0} can be written in a more convenient way by observing that 
\begin{align*}
	\sum_{r=-\infty}^{+ \infty } 
	\int_{- 2 r x } ^ { x - 2 r x } 
	e^ {   -  \frac{w ^2}{2t}  + \mu w}
	\, \mathrm d w 
	&=
	\sum_{  r = 0 } ^ {+ \infty}  
	\int_{ 2 r x } ^ { x + 2 r x } 
	e^ {   -  \frac{w ^2}{2t}  + \mu w}
	\, \mathrm d w 
	+ 
	\sum_{  r = 1 } ^ {+ \infty}  
	\int_{-  2 r x } ^ { x  -  2 r x } 
	e^ {   -  \frac{w ^2}{2t}  + \mu w}
	\, \mathrm d w 
	\\
	&=
	\int_0^\infty e^{ - \frac{w^2}{2t} } g (w, x) \,\mathrm d w 
	\\
\end{align*}
where 
\begin{align*}
g(w, x)
&= e^{ \mu w (-1)^{ \left \lfloor \frac w x \right \rfloor} }
\end{align*}
and $  \left \lfloor  z \right \rfloor  $ denotes the integer part of the real number $ z $.
In conclusion
\begin{align*}
&
P\bigg\{ \max_{0 \leq z \leq t} B^\mu(z) <x \Big |  \inf_{0 < z < t} B^\mu (z) > 0, B^\mu(0) = 0  \bigg\}
\label{eq:max-bmd-t-v0-fin}
\numberthis
\\
&=
\frac{  
	\sum_{r = -\infty} ^ { + \infty}
	(-1)^r
	e^{  -  \mu r x - \frac{x^2 r^2}{2t}  } 
	+ \mu
	\int_0^\infty e^{ - \frac{w^2}{2t} } e^{ \mu w (-1)^{ \left \lfloor \frac w x \right \rfloor} } \,\mathrm d w
}{
	1 + \mu \int_0^\infty 
	e^{  - \frac{y^2}{2t}  } 
	e^{ \mu y } \,\mathrm d y	
} \,\, .
\end{align*}
Observe that  
\begin{enumerate}[(i)]
\item for $ \mu = 0 $ we obtain the well-known distribution of the driftless meander
\begin{equation}\label{eq:max-bmd-mu0}
P\bigg\{ \max_{0 \leq z \leq t} B(z) <x \Big |  \inf_{0 < z < t} B (z) > 0, B(0) = 0  \bigg\}
=
\sum_{r = - \infty} ^ { + \infty} 
(-1)^r
e^{  - \frac{x^2 r^2 }{2t}  } \,\, .
\end{equation}
\item 
For $  x \to \infty  $, $  \mu \neq 0 $, we obtain that  \eqref{eq:max-bmd-t-v0-fin} tends to 1
as can be seen from \eqref{eq:max-bmd-t-v}.
\item 
In order to prove that \eqref{eq:max-bmd-t-v0-fin} for $  x \to 0  $, $  \mu \neq 0 $ 
we write 
\[
\sum_{k=-\infty}^{+ \infty} (-1)^k e^{- \frac{x^2 k^2}{2t} + \mu k x } 
= 
\sum_{k=-\infty}^{+ \infty} (-1)^k e^{- \frac{x^2 k^2}{2t} } 
\left(
1 + \int_0^{\mu k x} e^w \, \mathrm d w 
\right). 
\]
The first term tends to zero because we can apply
\[
\sum_{k=-\infty}^{+ \infty} (-1)^k e^{- \frac{x^2 k^2}{2t} } 
= 
\frac{\sqrt{4 \pi t}}{x} 
\sum_{k=1}^{+ \infty} e^{- \frac{ (2k-1)^2 \pi^2 t }{4 x^2} } 
\]
(see \cite{durbin}, eq. (3.4.9)).
\item 
Note that for $  \mu = 0 $ we have that  \eqref{eq:max-bmd-mu0} coincides with the
distribution of the maximum of the Brownian bridge, that is
\begin{equation*}
P\bigg\{ \max_{0 \leq z \leq t} B(z) <x \Big |  \inf_{0 < z < t} B (z) > 0, B(0) = 0  \bigg\}
= 
P\bigg\{ \max_{0 \leq z \leq t} |B(z)| < \frac x2 \Big |  B(0) = 0, B(t) = 0  \bigg\} \,\, .
\end{equation*}
\end{enumerate}

\subsection{First passage times.}
We now study the distribution of the first passage time of a Brownian motion
 evolving 
under the condition that the minimum of the process is larger than $ v $ (eventually $ v=0 $) up to time $ t > 0$,
the process being free afterwards.
The first passage time is denoted  by $ T_x = \inf \{s < t' : B^\mu(s) = x\} $, $ t'>t $. 

For $ s <  t $ we infer  from \eqref{eq:max-bmd-ini}  that 
\begin{align*}
	&
	P\left\{ T_x > s \Big| \min_{ 0 \leq z \leq t }B^\mu( z) > v, B^\mu(0) = u \right\} 
	\numberthis \label{eq:first-psg-dist}
	\\
	&=
	\int_v^x P\left\{ \max_{ 0 \leq z \leq s }B^\mu( z) < x, \min_{ 0 \leq z \leq s }B^\mu( z) > v , 
	B^\mu ( s ) \in \mathrm d y  \Big | B^\mu ( 0) = u \right\}	\times 
	\\
	& \qquad \times 
	\frac{
		P\left\{  \min_{ s \leq z \leq t }B^\mu( z) > v \Big | B^\mu ( s) = y \right\}			
	}{
		P\left\{  \min_{ 0 \leq z \leq t }B^\mu( z) > v \Big | B^\mu ( 0) = u \right\}
	}
	\\
	&=
	\int_v^x \, \frac{\mathrm d y}{\sqrt{2 \pi s}} \sum_{k = - \infty}^{+\infty}
	\left[ 
	e^{    -   \frac{(y - u - 2k(x-v))^2}{2s} } 
	- 
	e^{    -   \frac{(2v - y - u +2k(x-v))^2}{2s} }  
	\right ] 
	e^{ - \frac{\mu^2s}{2} + \mu ( y - u - 2k(x-v))} 
	\,\,\times 
	\\
	& \qquad \times 
	\frac{
		P\Big\{   \min_{s \leq z \leq t} B^\mu(z)   > v  \Big | B^\mu ( s) = y \Big\}
	}{
		\int_v^\infty 
		\left[ 
		e^{  - \frac{(w-u)^2}{2t}  } 
		- 
		e^{  - \frac{(2v - w-u)^2}{2t}  } 
		\right]
		e^{  - \frac{\mu^2 t}{2} + \mu ( w - u)   } \, \frac{ \mathrm d w}{\sqrt{2 \pi t} } 
	}
\end{align*}
for $ x > v$. If $ s > t $ we have that
\begin{align*}
&
	P\left\{ T_x \in \mathrm d s \Big| \min_{ 0 \leq z \leq t }B^\mu( z) > v, B^\mu(0) = u \right\} 
	/ \mathrm d s
	\numberthis \label{eq:first-psg-dist-ts}
	\\
	&=
	\int_v^x 
	P\left\{ T_x\in \mathrm d s \Big| B^\mu(t) = y \right\} / \mathrm d s
	\,\,
	P\left\{  
	B^\mu(t) \in \mathrm d y, \max_{0 \leq z \leq t} B^\mu (z) < x \Big| \min_{ 0 \leq z \leq t } B^\mu(z) > v , B^\mu(0) = u
	\right\}
	\\
	&=
	\frac{
		\int_v^x
		\frac{ (x-y)\, e^{ - \frac{  |x- y - \mu (s-t)|^2} {2(s-t)} } }
		{\sqrt{2 \pi (s-t)^3} } 
		\sum_{k = - \infty}^{+\infty}
		\left[ 
		e^{    -   \frac{(y - u - 2k(x-v))^2}{2t} } 
		- 
		e^{    -   \frac{(2v - y - u +2k(x-v))^2}{2t} }  
		\right ] 
			e^{  \mu ( y  - 2k(x-v))}
		 \, \mathrm d y}
	{\int_v^\infty \left(  e^{- \frac{(w - u)^2}{2t}} -  e^{- \frac{(2v - w - u)^2}{2t}}  \right)
	e^{\mu w} \, \mathrm d w 
	} \,\, .
\end{align*}
%
%
For  $  u \to v $ and $ s < t $ the first-passage time becomes
\begin{align*}
&
	P\left\{ T_x > s \Big| \min_{ 0 \leq z \leq t }B^\mu( z) > v, B^\mu(0) = v \right\} =
	\numberthis \label{eq:first-psg-dist-v}
	\\
	&=
		\frac{ 
	\int_v^x \sum_{k = - \infty}^{+\infty}
	e^{    -   \frac{(y - v - 2k(x-v))^2}{2 s} } 
	\frac { (y - v - 2k (x-v) )}{s\sqrt{2 \pi s }}
	e^{ - \frac {\mu^2 s } 2 + \mu( y-v - 2 k \mu (x-v) ) }
	\,   
%
		%
		m(t;s,y) \,\mathrm d y
	}{
		\int_v^\infty 
		e^{  - \frac{(w-v)^2}{2t}  } 
		\frac{ (w-v)  }{\sqrt{2 \pi t} }
		e^{ \mu (w-v) + \frac{\mu^2 t}{2}   } \,  \mathrm d w
	} \,\,
\end{align*}
where 

\begin{align*} 
	&
	m(t;s,y) = 
 	P\Big\{   \min_{s \leq z \leq t} B^\mu(z)   > v  \Big | B^\mu ( s) = y \Big\}
 	\numberthis \\
 	&=
 	\int_v^\infty 
 	\left[ 
 	e^{  - \frac{(w-y)^2}{2(t-s)}  } 
 	- 
 	e^{  - \frac{(2v - w-y)^2}{2(t-s)}  } 
 	\right]
 	e^{  - \frac{\mu^2 (t-s)}{2} + \mu ( w - y) }  \frac{\mathrm d w}{\sqrt{2\pi(t-s)}} \,\, .
\end{align*}

In the case $  t < s< t' $, the limit for $ u \downarrow v $ in 
\eqref{eq:first-psg-dist-ts}
yields
\begin{align*}
&
P\left\{ T_x \in \mathrm d s \Big| \inf_{ 0 < z < t }B^\mu( z) > v, B^\mu(0) = v \right\} 
/ \mathrm d s
\numberthis \label{eq:first-psg-dist-ts-v}
\\
&=
\frac{\int_v^x 
	\frac{ x-y}
	{\sqrt{2 \pi (s-t)^3} } 
	e^{ - \frac{  |x- y - \mu (s-t)|^2} {2(s-t)} }
	\sum_{k = - \infty}^{+\infty}
	(y - v - 2k(x-v))
	e^{    -   \frac{(y - v - 2k(x-v))^2}{2t} } 
	e^{  \mu ( y  - 2k(x-v))}
	\, \mathrm d y}
{
	\int_0^\infty 
	(w-v) e^{- \frac{ (w - v)^2} { 2 t } + \mu w }  
	\, \mathrm d w
}\,\,.
\end{align*}
Some further simplification can be obtained in  \eqref{eq:first-psg-dist-v} and \eqref{eq:first-psg-dist-ts-v}
for $ v=0 $. For $ s < t  $ we have in particular that \eqref{eq:first-psg-dist-v}
further simplifies as 
\begin{align*}
&
P\left\{ T_x > s \Big| \inf_{ 0 < z < t }B^\mu( z) > 0, B^\mu(0) = 0 \right\} 
\numberthis \label{eq:first-psg-dist-0}
\\
&=
\int_0^x \sum_{k = - \infty}^{+\infty}
e^{    -   \frac{(y  - 2kx)^2}{2 s} } 
\frac { (y -  2k x )}{s\sqrt{2 \pi s }}
e^{ -2 k \mu x}
\,   \mathrm d y
\frac{ 
	\int_0^\infty
	\left[
	e^{  - \frac{(w - y )^2}{  2 (t-s) }    }
	- 
	e^{  - \frac{  (w + y )^2}{ 2 (t-s) }    } \right]
	\, \frac{ e^{ \mu w}  }{\sqrt{2 \pi (t-s)} }
	 \mathrm d w
}{
	\int_0^\infty 
	\frac w {t \sqrt{2 \pi t} }
	e^{  - \frac{w^2}{2t} + \mu w   } 	 
	\mathrm d w
} \,\,  .
\end{align*}


%

\section{Representation of the Brownian meander with drift}

For the non-drifted meander $ M(t), t>0 $ it is well known that the following representation holds
\begin{equation}\label{eq:mdr-repr}
\frac{
\Big|  B(T_0 + s (t - T_0)) \Big|
}{\sqrt{t - T_0}}
\stackrel{i.d.}{=} M(s) \qquad 0 < s < 1
\end{equation}
where $ T_0 =  \sup\{  s < t : B(s) = 0\} $ and $ B(t), t>0  $  is a standard Brownian motion (see \citet{pitman99}).

For the Brownian meander with drift we have instead the following result.

\begin{theorem}
	\begin{align}\label{eq:mdr-repr-thm}
	P\left\{   
	\frac { 
	\Big | B^\mu\Big ( T_0^\mu + s ( t - T_0^\mu) \Big )  \Big |    
	}
	{
	\sqrt{t - T^\mu_0}
	} \in \mathrm d y 
	\right\} = 
	\frac 12 
	\left[
	P\left\{ M^{- \mu \sqrt{ t - T_0^\mu}}(s) \in\mathrm d y \right\} + 
	P\left\{ M^{\mu \sqrt{ t - T_0^\mu}} (s)\in\mathrm d y \right\}  
	\right]
	\end{align}
for $  0< s < 1 $, where $ T_0^\mu $ is the last passage time through zero before $ t $ of $ B^\mu $ and has distribution
\begin{align*}
P( T _0^\mu \in \mathrm d a) / \mathrm d a &=
\frac{e^{- \frac{\mu^2 t }{2}}}{\pi \sqrt{a (t-a)}} + \frac{\mu^2}{2 \pi} \int_a^t \frac{e^{ - \frac{ \mu^2 a }{2 }}}{\sqrt{a ( y -a )}} 
\, \mathrm d y 
\numberthis 
\\
&=
\mathbb E \left( \frac{1 }{\pi \sqrt{a ( W -a )}} \mathds{1}_{W \geq a } \right) \qquad 0 <a<t
\end{align*}
where $  W  $ is a truncated exponential distribution with density
\begin{align*}
&f_W(w) = \frac{\mu^2}{2} e^{  - \frac{\mu^2}{2} w  } \qquad 0<w<t \\
&P(W = t ) = e^{-\frac{\mu^2 t}{2}}\, .
\end{align*}
\end{theorem}
\begin{proof}
	We start by observing that, for $  0< s<1 $
	\begin{align}
		P\left\{   
	\frac { 
		\Big | B^\mu\Big ( T_0^\mu + s ( t - T_0^\mu) \Big )  \Big |    
	}
	{
		\sqrt{t - T^\mu_0}
	} <  y 
	\right\} 
	= 
	\int_0^t P\left\{T_0^\mu \in \mathrm d z \right\} \int_{ -y \sqrt{t-z}}^{ y \sqrt{t-z}}
	P \left\{ 
	 B^\mu  ( T_0^\mu + s ( t - T_0^\mu) )  \in \mathrm d w \Big | T_0^\mu = z
	\right\}
	\end{align}
	and thus 
	\begin{align*}
	&
	P\left\{   
	\frac { 
		\Big | B^\mu\Big ( T_0^\mu + s ( t - T_0^\mu) \Big )  \Big |    
	}
	{
		\sqrt{t - T^\mu_0}
	} \in \mathrm d y 
	\right\} 
	\numberthis 
	\label{eq:mdr-repr-full}
	\\
	 &=
	 \int_0^t P\left\{T_0^\mu \in \mathrm d z \right\} \sqrt{t -z}
	 \big[
	 P \left\{ 
	 B^\mu ( z + s (t-z)) \in \mathrm d (  y \sqrt{t-z}) | T_0^\mu = z
 	 \right\} 
 	 \\
 	 & \qquad 
 	 +
 	 P \left\{ 
 	 B^\mu ( z + s (t-z)) \in \mathrm d ( - y \sqrt{t-z}) | T_0^\mu = z
 	 \right\} 
	 \big]
	\end{align*}
	Note that 
	\begin{align*}
	&
	P\left\{ 
	B^\mu ( z + s (t-z)) \in \mathrm d y \Big | (T_0^\mu = z) \cap 
	\Big ( 
	\big( 
	\inf_{z < l \leq t} B^\mu ( l ) > 0 
	 \big)
	 \cup 
	 \big( 
	 \sup_{z < l \leq t} B^\mu ( l ) < 0 
	 \big)
	 \Big)
	\right\}
	\numberthis 
	\\
	&=
	P\left\{ 
	B^\mu ( z + s (t-z)) \in \mathrm d y | (T_0^\mu = z) \cap 
	(C
	\cup 
	D)
	\right\}
	\\
	&=
	P\left\{ 
	B^\mu ( z + s (t-z)) \in \mathrm d y | (B^\mu(z) = 0) \cap 
	C
	\right\} 
	\frac{\displaystyle
		P(C | B^\mu(z) = 0 )}
	{
		\displaystyle
			P(C | B^\mu(z) = 0 )) + 
		P(D | B^\mu(z) = 0 ))}
	\\
	& 
	\qquad + 
	P\left\{ 
	B^\mu ( z + s (t-z)) \in \mathrm d y | (B^\mu(z) = 0) \cap 
	D
	\right\} 
	\frac{\displaystyle
		P(D | B^\mu(z) = 0 )}
	{
		\displaystyle
		P(C | B^\mu(z) = 0 )) + 
		P(D | B^\mu(z) = 0 ))}
	\end{align*}
	The ratio 
	$  \frac{P(D | B^\mu(z) = 0)}{P( C |B^\mu(z) = 0 )} $
	can be evaluated as the following limit
	\begin{align*}
	\frac{P(D | B^\mu(z) = 0)}{P( C |B^\mu(z) = 0 )} 
	&= 
	\lim_{u \downarrow 0 } 
	\frac{
		P ( \max_{z \leq l \leq t} B^\mu ( l ) < 0  | B^\mu(z) =  - u)
	}{P ( \min_{z \leq l \leq t} B^\mu ( l ) > 0  | B^\mu(z) =  u)}
	\\
	&=
	\lim_{u \downarrow 0 } 
	\frac{
		P ( \max_{0 \leq l \leq t - z } B^\mu ( l ) < 0  | B^\mu(0) =  - u)
	}{P ( \min_{0 \leq l \leq t -z } B^\mu ( l ) > 0  | B^\mu(0) =  u)}
	\\
	&=
	\lim_{u \downarrow 0 } 
	\frac{
		P ( \overline T_0^\mu > t -z   | B^\mu(0) =  - u)
	}{P ( \overline T_0^\mu > t -z   | B^\mu(0) =  u)}
	\\
	&=
	\lim_{u \downarrow 0 } 
	\frac{
		1 - \int_0^{t-z} 
		u \, \frac{
		e^{ - \frac{ (u - \mu s ) ^ 2 }{2 s } }	
	}{\sqrt{2 \pi s^3}} \, \mathrm d s
	}{
	1 - \int_0^{t-z} 
	u \, \frac{
		e^{ - \frac{ (u + \mu s ) ^ 2 }{2 s } }	
	}{\sqrt{2 \pi s^3}} \, \mathrm d s
	}
	=
	1 
	\end{align*}
	where
	$  \overline T _ 0 ^ \mu = \inf\{ s : B^\mu(s) = 0\} $ is the first passage time of a drifted 
	Brownian motion starting either above or below zero. 
	
	In view of formula \eqref{eq:bmd-drift-dist-prooved}, for $  v = 0, s = l - z, t $ replaced by $  t -z  $
	we obtain, by setting $  l = z + s ( t-z) $ that
	\begin{align*}
	&
	P\left\{ 
	B^\mu ( l) \in \mathrm d y | B^\mu(z)  = 0 ,  
	\inf_{z < w \leq t} B^\mu ( w ) > 0 
	\right\}
	\numberthis 
	\label{eq:mdr-repr-min}
	\\
	&=
	P\left\{ 
	B^\mu ( l - z ) \in \mathrm d y | B^\mu(z)  = 0 ,  
	\inf_{0 < w \leq t - z } B^\mu ( 0 ) > 0 
	\right\}
	\\
	&=
	\left ( \frac{t - z  }{l - z }\right ) ^{\!\frac 32}
	\frac{
		\mathrm d y \,\, 
		y e^{ - \frac{ y^2}{2 (l-z) }} 		
	}{
		\int_0^\infty 
		w 
		e^{- \frac{ w^2}{2 (t-z)}  + \mu w}    
		\, \mathrm d w 		
	}\,
	\int_0^\infty 
	\Big( 
	e^{- \frac{ ( w - y)^2}{2(t-l)}} 
	- e^{   - \frac{ (w+y)^2}{2 ( t-l)}   }
	\Big ) 
	\frac{  e^{ \mu w } } {\sqrt{2 \pi (t - l)}}
	 \, \mathrm d w
	\end{align*}
	In the same way we have that
	\begin{align*}
	&
	P\left\{ 
	B^\mu ( l) \in \mathrm d y | B^\mu(z)  = 0 ,  
	\sup_{z < w \leq t} B^\mu ( w ) < 0 
	\right\}
	\numberthis 
	\label{eq:mdr-repr-max}
	\\
	&=
	\left ( \frac{t - z  }{l - z }\right ) ^{\!\frac 32}
	\frac{
		\mathrm d y \,\, 
		y e^{ - \frac{ y^2}{2 (l-z) }} 		
	}{
		\int_0^\infty 
		w 
		e^{- \frac{ w^2}{2 (t-z)}  - \mu w}    
		\, \mathrm d w 		
	}\,
	\int_0^\infty 
	\Big( 
	e^{- \frac{ ( w - y)^2}{2(t-l)}} 
	- e^{   - \frac{ (w+y)^2}{2 ( t-l)}   }
	\Big ) 
	\frac{  e^{ - \mu w } } {\sqrt{2 \pi (t - l)}}
	\, \mathrm d w
	\end{align*}
	By inserting  \eqref{eq:mdr-repr-min} and  \eqref{eq:mdr-repr-max} 
	into \eqref{eq:mdr-repr-full}  we obtain 
	\begin{align*}
	&
	P\left\{   
	\frac { 
		\Big | B^\mu\Big ( T_0^\mu + s ( t - T_0^\mu) \Big )  \Big |    
	}
	{
		\sqrt{t - T^\mu_0}
	} \in \mathrm d y 
	\right\} 
	=
	\frac{1}{2}
	\int_0^t P \{ T_0^\mu \in \mathrm d z \}  \quad \times
	\numberthis 
	\label{eq:mdr-repr-full-ins}
	\\
	& \qquad \times 
	\Bigg[ 
	 \frac{1}{s^{\!\frac 32} } 
	\frac{
		\mathrm d y \,\, 
		y \sqrt{t-z} \, e^{ - \frac{ y^2}{2 s }} 		
	}{
		\int_0^\infty 
		w 
		e^{- \frac{ w^2}{2 (t-z)}  + \mu w}    
		\, \mathrm d w 		
	}\,
	\int_0^\infty 
	\Big( 
	e^{- \frac{ ( w - y \sqrt{t-z}   )^2  }{    2(t-z)(1-s)}} 
	- e^{   - \frac{ (w+y \sqrt{t-z} )^2}{2 ( t-z)(1-s)}   }
	\Big ) 
	\frac{  e^{ \mu w } } {\sqrt{2 \pi (1-s)}}
	\, \mathrm d w 
	\\
	& \qquad \hphantom{\Bigg[  }  + 
	\frac{1}{s^{\!\frac 32} } 
	\frac{
		\mathrm d y \,\, 
		y \sqrt{t-z} \, e^{ - \frac{ y^2}{2 s }} 		
	}{
		\int_0^\infty 
		w 
		e^{- \frac{ w^2}{2 (t-z)}  - \mu w}    
		\, \mathrm d w 		
	}\,
	\int_0^\infty 
	\Big( 
	e^{- \frac{ ( w - y \sqrt{t-z}   )^2  }{    2(t-z)(1-s)}} 
	- e^{   - \frac{ (w+y \sqrt{t-z} )^2}{2 ( t-z)(1-s)}   }
	\Big ) 
	\frac{  e^{ - \mu w } } {\sqrt{2 \pi (1-s)}}
	\, \mathrm d w
	\Bigg] 
	\end{align*}
	The transformation $ w = w' \sqrt{t-z} $ further simplifies \eqref{eq:mdr-repr-full-ins}
	as
	\begin{align*}
	&
	P\left\{   
	\frac { 
		\Big | B^\mu\Big ( T_0^\mu + s ( t - T_0^\mu) \Big )  \Big |    
	}
	{
		\sqrt{t - T^\mu_0}
	} \in \mathrm d y 
	\right\} 
	=
	\frac{1}{2}
	\int_0^t P \{ T_0^\mu \in \mathrm d z \}  \quad \times
	\numberthis 
	\label{eq:mdr-repr-full-ins-2}
	\\
	& \qquad \times 
	\Bigg[ 
	\frac{
		\mathrm d y \,\, 
		s^{ -\frac 32}
		y \, e^{ - \frac{ y^2}{2 s }} 		
	}{
		\int_0^\infty 
		w 
		e^{- \frac{ w^2}{2}  + \mu w\sqrt{t-z}}    
		\, \mathrm d w 		
	}\,
	\int_0^\infty 
	\Big( 
	e^{- \frac{ ( w - y    )^2  }{    2(1-s)}} 
	- e^{   - \frac{ (w+y  )^2}{2 (1-s)}   }
	\Big ) 
	\frac{  e^{ \mu w \sqrt{t-z}} } {\sqrt{2 \pi (1-s)}}
	\, \mathrm d w 
	\\
	& \qquad \hphantom{\Bigg[  }  + 
	\frac{
		\mathrm d y \,\, 
		s^{ -\frac 32}
		y \, e^{ - \frac{ y^2}{2 s }} 		
	}{
		\int_0^\infty 
		w 
		e^{- \frac{ w^2}{2 }  - \mu w\sqrt{t-z}}    
		\, \mathrm d w 		
	}\,
	\int_0^\infty 
	\Big( 
	e^{- \frac{ ( w - y   )^2  }{    2(1-s)}} 
	- e^{   - \frac{ (w+y  )^2}{2 (1-s)}   }
	\Big ) 
	\frac{  e^{ - \mu w \sqrt{t-z} } } {\sqrt{2 \pi (1-s)}}
	\, \mathrm d w
	\Bigg] 
	\end{align*}
	We recognize  inside \eqref{eq:mdr-repr-full-ins-2} the distribution of a drifted Brownian meander
	whose drift can be written as 
	\[
	\mu X \sqrt{t - T^\mu_0} 
	\]
	where $ X $ is a two-valued, symmetric r.v. independent from $ T^\mu_0  $ which is 
	the last-passage time through zero, before $ t $, of  the Brownian motion $ B^\mu $. 
	In conclusion we have that 
	\begin{align*}
	P\left\{   
	\frac { 
		\Big | B^\mu\Big ( T_0^\mu + s ( t - T_0^\mu) \Big )  \Big |    
	}
	{
		\sqrt{t - T^\mu_0}
	} \in \mathrm d y 
	\right\} 
	=
	P\left\{ 
	M^{\mu X \sqrt{t - T^\mu_0} } ( s) \in \mathrm d y
	\right\}
	\end{align*}
\end{proof}
\begin{remark}
	For $ \mu = 0 $ result \eqref{eq:mdr-repr-thm} coincides with \eqref{eq:mdr-repr}, valid 
	for the non-drifted Brownian meander.
\end{remark}


%
%

\section{Excursion with drift}

We now present some results about the Brownian excursion with drift. 
We provide a construction which  is analogous to that of the drifted Brownian meander. 
We consider the sequence of conditional processes $ \{ B^\mu | \Lambda_{u,v,c} : u,c >v \}_{} $ whose
sample paths lie in sets of the form 
$\Lambda_{u,v,c} = \{ \omega \in C[0,T]  :{\min_{ 0 \leq z \leq t } \omega(z) > v },\,\, \omega(0) = u,\,\, \omega(t) = c  \}$, with $ c,u > v $. We study the distribution of these processes and of process
$ B^\mu | \Lambda_{v,v,v} $ defined as the weak limit of $  B^\mu | \Lambda_{u,v,c} $ when the starting 
point $ u $ and the endpoint $ c $ collapse onto the barrier $ v $.

The main tool used to carry out the needed calculation is the distribution of an absorbing Brownian motion  
with distribution  \eqref{eq:bm-drift-recall-new}. With this at hand we can write
\begin{align*}
	&
	P \Big\{ 
	B^\mu(s) \in \mathrm d y \Big | \min_{ 0 \leq z \leq t } B^\mu(z) > v, B^\mu(0) = u, B^\mu(t) = c
	\Big\}
	 = \numberthis \label{eq:bmd-bdg-dft}\\
	&= 
		\frac{
			e^{- \frac{( y- \mu s -u)^2}{2s} } - 
			e^{ - 2 \mu u}
			e^{- \frac{( y + u - 2 v - \mu s )^2}{2s} }
		}{\sqrt{2\pi s}} \cdot 
		\frac{
			e^{- \frac{( c- \mu(t-s) - y)^2}{2(t-s)} } - 
			e^{ - 2 \mu y} 
			e^{- \frac{(c + y - 2 v  - \mu (t-s) )^2}{2(t-s)} }
		}{\sqrt{2\pi (t-s)}} 
		%
	%
	%
	\times \\
	&
	\qquad \times 
	\left( 
		\frac{
			e^{- \frac{( c- \mu t -  u)^2}{2t} } -
			e^{-2\mu u } 
			e^{- \frac{( c + u - 2 v  - \mu t )^2}{2t} }
		}{\sqrt{2 \pi t}}
	\right)^{\!-1}	
	\,  \mathrm d y \\
	&=
	\sqrt{\frac{t}{2 \pi s(t-s)}}	\!
		%
		\Big(
		e^{- \frac{( y- u)^2}{2s} } -
		e^{- \frac{( 2v - y   - u)^2}{2s} }
		\Big)
		e^{ - \frac{\mu^2 s}{2} + \mu (y-u)}
	%
	%
	%
		%
	\Big(
	e^{- \frac{( c -y)^2}{2(t-s)} } - 
	e^{- \frac{( 2v - c  - y)^2}{2(t-s)} }
	\Big)
	e^{- \frac{ \mu^2(t-s)}{2} + \mu(c-y) }
	\\
	& \qquad \times
	\left[ 
	\left( 
	e^{- \frac{( c - u)^2}{2t} } - 
	e^{- \frac{( 2v -c - u)^2}{2t} }
	\right)
	e^{ - \frac{\mu^2 t}{2} + \mu (c-u)}
	\right]^{\!-1}
	  \mathrm d y
\end{align*}

It is straightforward to check that in \eqref{eq:bmd-bdg-dft} the 
resulting distribution does not depend on the value of the drift $ \mu $. This consideration holds
for any finite dimensional distribution of the drifted excursion, as stated in the following theorem.

\begin{theorem}\label{thm:exc-fin-dim}
	The finite dimensional distributions of the drifted Brownian excursion process $ B^\mu | \Lambda_{u,v,c}  $ are given by  
\begin{align*}
&
P \bigg\{ \bigcap_{j=1}^n \left( B^\mu(s_j) \in \mathrm d y_j \right) \,\Big \vert \min_{0\leq z \leq t} B^\mu(z)> v , B^\mu(0)=u, B^\mu(t) = c\bigg \}  = \numberthis  \label{eq:exc-fin-dim}\\ 
&   \prod_{j=1}^{n}  
\Bigg \{ 
\frac{ e^{ - \frac{(y_j - y _{j-1})^2}{2(s_j - s_{j-1})}  } - 
e^{ - \frac{(2v - y_j - y _{j-1})^2}{2(s_j - s_{j-1})}  } }   {\sqrt{2 \pi ( s_j - s_{j-1})} } \,\mathrm d y_j
\Bigg\} 
	\frac{ e^{ - \frac{ (c - y_n) ^ 2}{2(t-s_n)} }
		- e^{ - \frac{ (2v - c  - y_n) ^ 2}{2(t-s_n)}} }{\sqrt{2 \pi (t - s_{n})}}
\left(
	\frac{ e^{ - \frac{ (c - u) ^ 2}{2t} }
		- e^{ - \frac{ (2 v - c  -  u) ^ 2}{2t}} }{\sqrt{2 \pi t}}
\right)^{\!\!-1}
\end{align*}
for $ 0 = s_0 < s_1 < \cdots < s_n < t, y_0 = u $ with $ y_j,u,c > v $, and the following equality in distribution holds
\begin{equation}\label{eq:exc-no-drift-eq}
B^\mu | \Lambda_{u,v,c} \stackrel{law}{=}   B | \Lambda_{u,v,c}. 
\end{equation}

\end{theorem}

\begin{proof}
	By Markovianity we can write that
\begin{align*}
&
P \bigg\{ \bigcap_{j=1}^n \left( B^\mu(s_j) \in \mathrm d y_j \right) \,\Big \vert \min_{0\leq z \leq t} B^\mu(z)> v , B^\mu(0)=u, B^\mu(t) = c\bigg \}  = \\ 
& =  \prod_{j=i}^{n}  
\Bigg \{ 
\frac{ e^{ - \frac{(y_j - y _{j-1})^2}{2(s_j - s_{j-1})}  } - 
	e^{ - \frac{(2v - y_j - y _{j-1})^2}{2(s_j - s_{j-1})}  } }   {\sqrt{2 \pi ( s_j - s_{j-1})} }
e^{ - \frac{ \mu^2}{2} (s_j - s_{j-1}) +    \mu(y_j - y_{j-1})   }
\Bigg\} 
\times 
\\
&\qquad \times 
\frac{ e^{ - \frac{ (c - y_n) ^ 2}{2(t-s_n)} }
	- e^{ - \frac{ (2v - c  - y_n) ^ 2}{2(t-s_n)}} }{\sqrt{2 \pi (t - s_{n})}}
e^{ - \frac{ \mu^2}{2} (t - s_{n}) + \mu(c - y_{n}) }
\left(
\frac{ e^{ - \frac{ (c - u) ^ 2}{2t} }
	- e^{ - \frac{ (2 v - c  -  u) ^ 2}{2t}} }{\sqrt{2 \pi t}}
e^{- \frac{ \mu^2}{2} t + \mu(c - u)  }
\right)^{\!\!-1}
\end{align*}
thus the finite dimensional distributions of the excursion with drift coincide with the corresponding finite
dimensional distributions of the driftless Brownian excursion. 
\end{proof}
\begin{remark}
	
We observe that in analogy with result   \eqref{eq:exc-no-drift-eq} 
\citet{Beghin1999} showed that the distribution
\[
P\left\{  \max_{ 0 \leq s \leq t} B^\mu(s) > \beta \Big | B^\mu(t) = \eta \right\} 
=
e^{   -  \frac{ 2 \beta (\beta - \eta) }{t}} 
\]
is independent of $\mu$ for every $\eta $.
\end{remark}

%
%
%
Weak convergence to the Brownian excursion has been established in \cite{durrett77}. 
\begin{theorem}\label{thm:weak-conv-exc}
	The following weak limit holds:
	\begin{equation}\label{eq:weak-conv-exc}
	B^\mu \Big |\Lambda_{u,v,c} 
 \xRightarrow[u,c\downarrow v]{}
	B^\mu \Big | \Big \{ \inf_{ 0 < z < t } B^\mu(z) > v, B^\mu(0) = v, B^\mu(t) =c \Big \}
	\end{equation}
\end{theorem}
%

We give the explicit form of the finite-dimensional distributions of the right member of 
\eqref{eq:weak-conv-exc} in the special case when $ v =0 $. 
For $u\to 0$ the distribution \eqref{eq:exc-fin-dim} becomes
\begin{align*}
&
	P \bigg\{ \bigcap_{j=1}^n \left( B^\mu(s_j) \in \mathrm d y_j \right) \,\Big \vert \inf_{0< z < t} B^\mu(z)> 0 , B^\mu(0)=0, B^\mu(t) = c\bigg \} \\
	&=
	\lim_{ u \downarrow 0}
	P \bigg\{ \bigcap_{j=1}^n \left( B^\mu(s_j) \in \mathrm d y_j \right) \,\Big \vert \min_{0\leq z \leq t} B^\mu(z)> 0 , B^\mu(0)=u, B^\mu(t) = c\bigg \}
	\\
	& =
	\frac{y_1\, t \sqrt t }{c\, s_1 \sqrt{s_1}} e^{- \frac{y_1^2}{2 s_1} + \frac{c^2}{2t}}
	\prod_{j=2}^{n} 
	\left[
	\frac{
		e^{ - \frac{  (y_j - y_{j-1})^2   }{   2(s_j - s_{j-1})   }      } 
		- 
		e^{ - \frac{  (y_j + y_{j-1})^2   }{   2(s_j - s_{j-1})   }      } 
	}
	{
		\sqrt{2\pi ( s_j - s_{j-1})}
	}
	\right]
	\frac{
		e^{ - \frac{  (c - y_{n})^2   }{   2(t - s_{n})   }      } 
		- 
		e^{ - \frac{  (c + y_{n})^2   }{   2(t - s_{n})   }      } 
	}
	{
		\sqrt{2\pi ( t - s_{n})}
	}
\end{align*}
%

%
%
%
The further limit for $c \to 0 $ yields
\begin{align*}
&
P \bigg\{ \bigcap_{j=1}^n \left( B^\mu(s_j) \in \mathrm d y_j \right) \,\Big \vert \inf_{0< z < t} B^\mu(z)> 0 , B^\mu(0)=0, B^\mu(t) = 0\bigg \} \\
&=
\lim_{ u,c \downarrow 0}
P \bigg\{ \bigcap_{j=1}^n \left( B^\mu(s_j) \in \mathrm d y_j \right) \,\Big \vert \min_{0\leq z \leq t} B^\mu(z)> 0 , B^\mu(0)=u, B^\mu(t) = c\bigg \}
\\
& =
\frac{y_1\, t \sqrt t }{s_1 \sqrt{s_1}} e^{- \frac{y_1^2}{2 s_1}}
\prod_{j=2}^{n} 
\left[
\frac{
	e^{ - \frac{  (y_j - y_{j-1})^2   }{   2(s_j - s_{j-1})   }      } 
	- 
	e^{ - \frac{  (y_j + y_{j-1})^2   }{   2(s_j - s_{j-1})   }      } 
}
{
	\sqrt{2\pi ( s_j - s_{j-1})}
}
\right]
\frac{2 y_n}{t- s_n}
\frac{
	e^{ - \frac{   y_{n}^2   }{   2(t - s_{n})   }      } 
}
{
	\sqrt{2\pi ( t - s_{n})}
}.
\end{align*}

The one dimensional distribution of the excursion reads

\begin{align*}
& P\Big\{ B^\mu(s) \in \mathrm d y \Big |  \inf_{ 0 < z < t } B^\mu(z) > 0, B^\mu(0) = 0, B^\mu(t) =0 \Big \} = \numberthis \label{eq:bmd-bdg-dft-0-0} \\
&=
\sqrt \frac{2}{\pi} y^2 \left(\frac{t}{s(t-s)}\right) ^\frac 32 
e^{- \frac{ y^2t}{ 2 s (t-s)}} \, \mathrm d y \qquad y>0 \,,\,\, s<t.
\end{align*}

\subsection*{Acknowledgments}
We thank both referees for their accuracy in the analysis of the first draft of this paper.
They have detected misprints and errors and their constructive criticism has substantially 
improved the paper. 

\bibliographystyle{abbrvnat}
\bibliography{biblio}

\end{document}